%% file: main.tex
\documentclass[a4paper]{article}

\usepackage[a4paper, margin=1.3in]{geometry}

\usepackage{graphicx}
\usepackage{subcaption}
\usepackage{float}

\usepackage{amsmath}
\usepackage{amsthm}
\usepackage{amssymb}
\usepackage{amsfonts}
\usepackage{amsxtra}

\usepackage{textcomp}
\usepackage[T1]{fontenc}
\usepackage[utf8]{inputenc}
\usepackage{wasysym}
\usepackage{stmaryrd}
\usepackage[all]{xy}
\usepackage{bbm}

\usepackage{booktabs}
\usepackage{multirow}

\usepackage{xcolor}
\usepackage{fancyhdr}
\usepackage{enumitem}
\usepackage{comment}
\usepackage{caption}

\usepackage{algorithm}
\usepackage{algpseudocode}

\usepackage{listings}

\usepackage[numbers]{natbib}
\usepackage{hyperref}

\PassOptionsToPackage{section,above,below}{placeins}
\usepackage{placeins}

\newcommand{\RandWSQR}{{SE-QRCS}}
\newcommand{\sM}{{d}}
\newcommand{\sss}{{s}}

\newcommand{\kk}{{k'}}

\newtheorem{definition}{Definition}[section]
\newtheorem{proposition}[definition]{Proposition}
\newtheorem{remark}[definition]{Remark}

\newtheorem{theorem}[definition]{Theorem}
\newtheorem{lemma}[definition]{Lemma}
\newtheorem{corollary}[definition]{Corollary}

\title{Efficient QR-based Column Subset Selection through Randomized Sparse Embeddings }
\author{
    Israa Fakih\thanks{PSI Center for Scientific Computing, Theory and Data, Villigen PSI; Institute of Mathematics, EPFL, Switzerland. Email: israa.fakih@psi.ch} 
    \and
    Laura Grigori\thanks{PSI Center for Scientific Computing, Theory and Data, Villigen PSI; Institute of Mathematics, EPFL, Switzerland. Email: laura.grigori@epfl.ch}
}

\date{}  

\begin{document}

\maketitle
\begin{abstract}
In this paper, we introduce an efficient algorithm for column subset selection that combines the column-pivoted QR factorization with sparse subspace embeddings. The proposed method, \RandWSQR{}, is particularly effective for wide matrices with significantly more columns than rows. Starting from a matrix $A$, the algorithm selects $k$ columns from the sketched matrix  $B = A \Omega^T$, where $\Omega$ is a sparse oblivious subspace embedding for a subspace of dimension $rank(A)$. The sparsity structure of $\Omega$ is then exploited to map the selected pivots back to the corresponding columns of $A$, which are then used to produce the final subset of selected columns. We prove that this procedure yields a factorization with strong rank-revealing properties, thus revealing the spectrum of $A$. The resulting bounds exhibit a reduced dependence on the number of columns of $A$ compared to those obtained from the strong rank-revealing QR factorization of $A$. For general matrices, the algorithm can be extended by first applying an additional subspace embedding of $range(A)$.

\end{abstract}

\vspace{12pt}

\input{Introduction}
\input{Preliminaries}
\input{RandWSQR_Algorithm}
\input{Leverage_scores_case}
\input{Numerical_results}

\input{conclusion}

\bibliographystyle{plainnat}
 \nocite{*}
 \bibliography{references}

\end{document}

%% file: Introduction.tex
\section{Introduction}

In this paper we consider the problem of selecting $k$ columns forming a matrix $C$ from a large matrix $A\in \mathbb{R}^{\sM\times n}$, with $k\ll n$, that allow us to reveal its spectrum. This problem arises in many applications in scientific computing and data analysis and allows for example to reveal the rank of a matrix or compute its low rank approximation. This problem is also known as the Column Subset Selection Problem (CSSP), in particular when the goal is to minimize the norm of the error matrix $\|A-C C^+ A\|$, where $C^+$ denotes the pseudoinverse of $C$ and the spectral norm or the Frobenius norm are in general used. Finding the optimal $k$ columns is known to be an NP-hard problem \cite{SHITOV202152} when the Frobenius norm is used in this minimization problem.
The strong rank-revealing QR factorization \cite{gu1996efficient} can be used to select $k$ columns from a matrix $A$ while revealing its spectrum.  It has the following form:
\begin{equation}
A \Pi=Q\begin{pmatrix}
    R_{11} & R_{12}\\
     & R_{22}
\end{pmatrix}
\end{equation}
where $Q\in \mathbb{R}^{\sM\times \sM}$ is an orthogonal matrix, $R_{11}\in \mathbb{R}^{k\times k}$ is an upper triangular, $R_{12}\in \mathbb{R}^{k\times (n-k)}$, and $R_{22}\in \mathbb{R}^{(\sM-k)\times (n-k)}$. The column permutation matrix $\Pi\in \mathbb{R}^{n\times n}$ is chosen such that the following relations are satisfied:
\begin{equation}
\label{eq:strongRRQRIntro}
    1\leq \frac{\sigma_{i}(A)}{\sigma_{i}(R_{11})} ,\frac{\sigma_{j}(R_{22})}{\sigma_{j+k}(A)}\leq \rho_{1}(k,n) \qquad\qquad
    \left\| R_{11}^{-1} R_{12} \right\|_{\max} \leq \rho_{2}(k,n)
\end{equation}
for $1\leq i \leq k$ and $1\leq j \leq \min(m,n)-k$, with $\rho_{1}(k,n)$, $\rho_{2}(k,n)$ low-degree polynomials in $n$ and $k$. The algorithm introduced in \cite{gu1996efficient} satisfies \eqref{eq:strongRRQRIntro} for a given small constant $f > 1$ with $\rho_{1}(k,n) = \sqrt{1+f^{2}k(n-k)}$ and $\rho_2 (k,n) = f$. These inequalities ensure that the singular values of $R_{11}$ and $R_{22}$ are good approximations for the largest $k$ and last $\min(\sM,n)-k$ singular values of $A$ respectively. Writing the inequalities using ratios assumes for simplicity and without loss of generality that the singular values of $A$ and $R_{11}$ are nonzero.  This factorization also bounds the largest elements of  $R_{11}^{-1} R_{12}$, thus ensuring numerical stability of the factorization. If the factorization $A\Pi=QR$ satisfies the strong rank revealing properties in \eqref{eq:strongRRQRIntro}, then it can be shown that 
$$\|A-C C^+A\|_{2}\leq \rho_1(k,n)\|A-A_{k}\|_{2} = \rho_1(k,n) \sigma_{k+1} (A),$$
where $C$ consists of the first $k$ columns of $A \Pi$ and $A_{k}$ is the best rank-k  approximation of $A$ obtained through the truncated SVD factorization. The computational complexity of this factorization is $O(\sM nk)$, which makes it significantly expensive when dealing with very large datasets.

Several randomized methods have emerged in the recent years to accelerate solving  the column selection problem using the QR factorization with column permutations.  These approaches rely on a dimensionality reduction technique that allows one to embed a high dimensional subspace $\mathcal{W}\subset \mathbb{R}^{n}$ into a lower dimensional one $(\subseteq \mathbb{R}^{\ell})$ through a linear map $\Omega\in \mathbb{R}^{\ell \times n}$ with $\ell\ll n$ while preserving approximately the geometry of $\mathcal{W}$. In particular, inner products (and hence norms) between vectors in $\mathcal{W}$ are preserved up to a distortion parameter $\epsilon$. 
A linear map satisfying this property for any fixed subspace $\mathcal{W}$ of dimension $\sM$ with high probability is said to be an oblivious subspace embedding. In this work, we focus on sparse embeddings, where the embedding matrix has only a few nonzero entries. An example of oblivious sparse embedding is the OSNAP distribution, introduced in \cite{nelson2013osnap}. We then discuss the possibility of using non-oblivious sparse embeddings, where the embedding distribution is constructed based on the subspace $\mathcal{W}$, so that it preserves the norm of vectors for that specific subspace. 

Randomized QRCP \cite{martinsson2017householder,Duersch_2017,xiao2017fast}  consists of first sketching the columns of $A$, using an $\epsilon$-embedding of the range of $A$, to form the matrix $A_{sk}=\Omega A \in \mathbb{R}^{\ell\times n}$, where $\Omega\in \mathbb{R}^{\ell\times \sM}$ and $ \ell\ll \sM$. The columns are then selected by performing QR factorization with greedy column pivoting (QRCP) \cite{golub1965numerical} on the sketched matrix $A_{sk}$.  These columns are then used as pivots to compute the QR factorization of $A$. It has recently been shown \cite{grigori2025randomized} that if a strong RRQR factorization is used to select columns from the sketched matrix, then a factorization of $A$ that satisfies the strong RRQR properties in \eqref{eq:strongRRQRIntro} can be obtained under certain conditions. This approach allows to accelerate the QR factorization with column permutations by separating the selection of the pivots from the QR factorization of $A$ that can be computed without pivoting, thus allowing the usage of BLAS3 kernels and reducing the communication cost due to permuting columns. However, when the number of columns is very large, the selection of columns from $A_{sk}$ can still be very expensive or even dominate the overall cost.
 A different approach, introduced in \cite{boutsidis2009improved}, relies on sampling columns according to probabilities derived from the top$-k$ right singular vectors of $A$. The algorithm consists of two stages. The first stage, random, consists of sampling $O(klog(k))$ columns according to probabilities proportional to the row norms of the matrix containing the top-$k$ singular vectors. Examples of such probabilities are given in \cite{boutsidis2009improved} [Eqs (3.3) and (3.4)]. The second stage, deterministic, consists of applying any deterministic algorithm to sample $k$ columns from the subset chosen in the first stage. In \cite{boutsidis2009improved} it is shown that the approximation obtained by this algorithm when using Algorithm 1 of \cite{pan2000existence} in the second stage, satisfies the error bound $\|A-CC^{\dagger}A\|_{2}\leq \rho_1(k,n)\|A-A_{k}\|_{2}$ with $\rho_1(k,n)=O\left(k^{\frac{3}{4}}(min(\sM,n)-k)^{\frac{1}{4}}\right)$ with probability at least 0.7. This spectral error bound improves upon the corresponding approximation guarantee of the strong rank revealing QR factorization of Gu and Eisenstat \cite{gu1996efficient}. However, as computing the probabilities is very expensive, the time complexity of this algorithm is $O(\min (\sM n^{2},\sM^{2}n))$

In this paper, we introduce a randomized algorithm for selecting columns from a matrix
$A \in \mathbb{R}^{d\times n}$ that relies on sparse embeddings such as OSNAP
\cite{nelson2013osnap}, where the embedding matrix $\Omega \in \mathbb{R}^{\ell\times n}$
has exactly $s$ nonzero entries in each column, corresponding to appropriately scaled random signs, such that $\Omega$ satisfies the $\epsilon-$embedding property. The focus of this method is particularly matrices that have many more columns than rows, that is, with $d \ll n$. We refer to our algorithm as \RandWSQR{}.  Thus, it can be used to accelerate the column selection step in algorithms such as randomized QRCP.  Our column selection method is based on four stages. First, the row space of $A$ is embedded through a sparse embedding $\Omega \in \mathbb{R}^{\ell\times n}$ with $\ell\ll n$, to form the matrix $B=A\Omega^{T}\in \mathbb{R}^{\sM \times \ell}$ with a reduced number of columns.  In the second stage, k columns are selected from $B$ by computing its strong RRQR factorization.   Using the structure of sparse embeddings that have few nonzero entries in each row, the third stage relies on the fact that each column of the sketched matrix $B$ is a linear combination of a few columns of the original matrix $A$.  Thus, the columns selected by the strong RRQR factorization of  $B$ correspond to a subset of columns in the original matrix, that we denote by $\tilde{A}_{1}\in \mathbb{R}^{d\times p}$. The final $k$ columns are selected by computing the strong RRQR factorization of $\tilde{A}_{1}$.  

We show that, if strong RRQR is used to select columns from both $B$ and $\tilde{A}_{1}$ with a constant $f>1$, the resulting \RandWSQR{} factorization of $A$ is a strong rank revealing one satisfying 
\begin{equation}
\label{inequality_intro}
1 \leq \frac{\sigma_{i}(A)}{\sigma_{i}(R_{11})}, \;
\frac{\sigma_{j}(R_{22})}{\sigma_{j+k}(A)} \leq \rho_{1}(k,n),
\quad
\| R_{11}^{-1} R_{12} \|_{2} \leq \rho_{2}(k,n).
\end{equation}
with 
$$\rho_{1}(k,n)=\sqrt{1+\frac{4\omega^{*}}{(1-\epsilon)}(1+f^{2}k(\ell-k))(1+f^{2}k(p-k))}$$
and 
$$\rho_{2}(k,n)=\sqrt{\frac{2\omega^{*}}{(1-\epsilon)}(1+f^{2}k(\ell-k))(1+f^{2}k(p-k))}$$
where the value of $p$, the number of columns of $\tilde{A}_{1}$, is in expectation $\mathbb{E}(X)=n\left[1-\left(1-\frac{k}{\ell}\right)^{s}\right]$, which reduces to $\frac{nk}{\ell}$ for the CountSketch case (s=1), and $\omega^{*}=O(\frac{\sqrt{s}p}{\ell})$.
In terms of computational complexity, our algorithm relies on applying strong RRQR to two matrices, $B$ and $\tilde{A}_{1}$, that are smaller than $A$. Consequently, their complexity is $\mathcal{O}(\sM k(p+ \ell))$. This expression highlights a tradeoff between the embedding dimension $\ell$ and the expected number of selected columns $p$. Increasing $\ell$ increases the cost of factorizing $B$ while reducing $p$, and therefore decreasing the cost of factorizing $\tilde{A}_{1}$, and vice versa. In the CountSketch case, we have $\mathbb{E}(p)=\frac{nk}{\ell}$, which leads to a complexity of $\mathcal{O}(\sM k(\frac{nk}{\ell}+ \ell))$ for our algorithm. Minimizing this expression yields the choice $\ell=\sqrt{nk}$.  Combining this requirement with the sketch dimension needed for the embedding property leads to $\ell=\max(d^{2},\sqrt{nk})$ . By adding the complexity of sketching $A$, $\mathcal{O}(\sss n\sM)$, which is facilitated due to the sparsity of $\Omega$, the overall complexity for computing the pivots remains much lower than the traditional strong RRQR factorization of $A$. This is mainly because $p$ is smaller than n. This aspect is further validated in the numerical results that show a reduction in runtime with a factor greater than $7$ and $10$. 

We also discuss the case when leverage scores are known, as in the case of orthogonal matrices, or can be approximated. In this case,  non-oblivious sparse embeddings can be used as the one introduced in \cite{chenakkod2024optimal}, the leverage score sparsified embedding with independent entries or the leverage score sparsified embedding with independent rows. The obtained factorization satisfies the strong RRQR property in \eqref{inequality_intro}  with
 $$\rho_{1}(k,n)=\rho_{2}=O\left(\frac{k\sqrt{\kk}\log^{4}(d)\left(\frac{d}{\epsilon^{2}}-\kk\right)^{1/2}}{l_{min}^{1/2}\epsilon^{4}(1-\epsilon)^{1/2}}\right),$$
 where $l_{\min}$ is the minimum column leverage score of $A$ . It can be noticed that the bounds provided by the two-stage algorithm \cite{boutsidis2009improved} are tighter, however, our result not only gives a bound on the 2-norm of the error matrix, but also gives bounds for the approximations of all the singular values of $A$. We summarize in Table \ref{comparison Table} the bounds obtained by strong RRQR, our algorithm and the two-stage algorithm from \cite{boutsidis2009improved}.

These results are further validated by numerical experiments on different types of
matrices, including well-conditioned, ill-conditioned, and real-world datasets. The
experiments present a comparison between the singular values $\sigma_i(R_{11})$ and
$\sigma_i(A)$, where $R_{11}$ is the factor resulting from \RandWSQR{} and QRCP
factorization of $A$. In addition, we present a summary of the minimum, maximum, and
median results of the ratio $\sigma_i(R_{11})/\sigma_i(A)$. \RandWSQR{} is evaluated
using both oblivious (OSE) and non-oblivious (LESS) sparse embeddings, and its
performance is further assessed across different values of the sparsity $s$ and the
embedding dimension $\ell$. These results show that although \RandWSQR{} does not give a
better approximation to the SVD, the singular values obtained closely match those
obtained by QRCP in all the matrices tested. \RandWSQR{} also supports several applications. First, it can be applied to low-rank approximation of general square matrices, by first sketching the rows of the matrix and then applying \RandWSQR{} to the resulting sketched matrix. Second, we include the
usage of \RandWSQR{} in the LU factorization with panel rank-revealing pivoting
(LU\_PRRP), introduced in \cite{khabou2013lu}. Finally, \RandWSQR{} can also be applied to the selection of columns of wide and
short matrices arising from the matricization of tensors, making it relevant to
tensor computations and multilinear algebra. It can similarly be applied to the
selection of rows of an orthogonal matrix, as required by the discrete empirical interpolation method (DEIM) \cite{carrel2025interpolatory}

\begin{table}[t]
    \centering
\begin{tabular}{|c|c|c|c|}
\hline
     & sRRQR \cite{gu1996efficient} & Two\_stage Algorithm \cite{boutsidis2009improved} & \RandWSQR{} \\
  \hline
  $\rho(k,n)$ & $O\left(\sqrt{k(n-k)}\right)$ & $O\left(k^{\frac{3}{4}}(\min(d,n)-k)^{\frac{1}{4}}\right)$ & $O\left(k\sqrt{(p-k)(\ell-k)}\right)$ \\
  \hline
  Time & $O(dnk)$  & $O(\min (d^{2}n,dn^{2}))$ & $O\left(d(p+\ell)k+\sM n log(\sM)\right)$ \\
  \hline
\end{tabular}
 \caption{Comparison of the spectral error bound and time complexity of sRRQR, two-stage algorithm \cite{boutsidis2009improved} and \RandWSQR{}. In the \RandWSQR{} column, $\ell$ is the embedding dimension, $p$ is the size of the reduced column set $\tilde{A}_{1}$, using an oblivious sparse embedding, we have $\ell=O\left(\frac{dlog(d)}{\epsilon^{2}}\right)$, $s=O\left(\frac{log(d)}{\epsilon}\right)$ and $\mathbb{E}(p)=n\left(1-\left(1-\frac{k}{\ell}\right)^{s}\right)$}. 
 \label{comparison Table}
\end{table}

The remainder of this paper is structured as follows. In Section 2, we review some preliminaries on rank-revealing QR factorization and randomization techniques. Section 3 details the algebra of the proposed algorithm \RandWSQR{}. Section 4 provides the theoretical guarantees for the proposed algorithm for both cases, using oblivious and non-oblivious sparse embedding. Section 5 presents the experimental results of \RandWSQR{} on different types of matrices.

%% file: Preliminaries.tex
\section{Preliminaries}
This section introduces first the notation used in this paper. It then covers the strong rank-revealing QR factorization and the essential concepts related to randomization techniques.
\subsection{Notations}
We denote the identity matrix of dimension $n \times n$ and the zero matrix of dimension $m \times k$ by  $I_{n\times n}$ and $0_{m\times k}$ respectively. For an invertible matrix $A\in \mathbb{R}^{\sM \times \sM}$ and a general matrix $B\in \mathbb{R}^{\sM\times n}$, $\omega_{i}(A)$ corresponds to the 2-norm of the $i$-th row of $A^{-1}$ and $\gamma_{j}(B)$ corresponds to the 2-norm of the $j$-th column of B. $A_{i,:}$ denotes the $i$-th row of $A$. For $1\leq k \leq \sM$, \(A^{(1)}\) and \( A^{(2)} \) denote the submatrix of A consisting of the first $k$ and last \( \sM - k \) rows respectively. The letter $\epsilon$ denotes a real number in $(0,1)$. 
The $i$-th singular value of A is denoted by $\sigma_{i}$, while $g_{i,j}$ denotes the value of a function $g$ at (i,j). The scalar product $\langle .,. \rangle$ is the standard Euclidean product, $\|A\|_{2}=\sigma_{max}(A)$ is the spectral norm of the matrix $A$ and $\|A\|_{F}$ is its Frobenius norm. \\
For an integer $n$, [n] denotes the set ${1,2,...,n}$ and $|\mathcal{S}|$ denotes the cardinal of any subset $\mathcal{S}$.

\subsection{Strong Rank Revealing QR Factorization}
\begin{definition}
Given a matrix $A \in \mathbb{R}^{\sM \times n}$ with $\sM \ll n$, its partial QR factorization with column pivoting is
\begin{equation}
\label{QR_fact}
A \Pi = QR = Q \begin{pmatrix}
    R_{11} & R_{12} \\
    0 & R_{22}
\end{pmatrix},
\end{equation}
where $1 \leq k \leq d$, $Q\in \mathbb{R}^{\sM\times \sM}$ is orthogonal, $R_{11}\in \mathbb{R}^{k\times k}$ is upper triangular, $R_{12}\in \mathbb{R}^{k\times (n-k)}$, $R_{22}\in \mathbb{R}^{(\sM-k)\times(n-k)}$, and $\Pi\in \mathbb{R}^{n \times n}$ is a permutation matrix. 
\end{definition}
The factorization \eqref{QR_fact} is said to be a rank-revealing factorization if it verifies the following property \cite{ipsen2003rank}:
\begin{equation*}
 1\leq \frac{\sigma_{k}(A)}{\sigma_{1}(R_{11})} ,\frac{\sigma_{1}(R_{22})}{\sigma_{k+1}(A)}\leq \rho_{1}(k,n),
 \end{equation*}
where $\rho(k,n)$ is a function bounded by a low degree polynomial in $k$ and $n$. This factorization reveals the rank by placing the most significant columns in $R_{11}$.
The QR factorization with column pivoting, referred to as QRCP, is a greedy algorithm that progressively selects columns from $A$ that increase the value of the determinant of $R_{11}$. At each iteration, the column with the largest Euclidean norm is selected from the remaining columns, permuted to the leading position, and then the elements below its diagonal are annihilated using Householder reflectors. The remaining columns are updated according to this orthogonal transformation. Although this algorithm works very well in practice, it fails for some matrices such as the Kahan matrix \cite{kahan1966numerical}. For these reasons, Gu and Eisenstat \cite{gu1996efficient} introduced an alternative algorithm to interchange columns such that the factorization satisfies a tighter bounds on singular values and provides additional control on the matrix $\|R_{11}^{-1}R_{12}\|_{max}$. This is known as the strong rank revealing QR factorization.
\begin{definition}[Strong rank-revealing QR factorization (sRRQR)]
\label{def:strongRRQR}
The QR factorization \eqref{QR_fact} is said to be strong rank revealing if it satisfies 
\begin{equation}
\label{str1}
    1\leq \frac{\sigma_{i}(A)}{\sigma_{i}(R_{11})} ,\frac{\sigma_{j}(R_{22})}{\sigma_{j+k}(A)}\leq \rho_{1}(k,n), 
    \end{equation}
    \begin{equation}
    \label{str2}
    \left\| R_{11}^{-1} R_{12} \right\|_{\max} \leq \rho_{2}(k,n),
    \end{equation}
for $1 \leq i \leq k$ and $1 \leq j \leq \sM-k$, where $\rho_{1}(k,n)$ and $\rho_{2}(k,n)$ are functions bounded by low-degree polynomials in \(k\) and \(n\).
\end{definition}
 We assume for simplicity that $\sigma_{i}(A)\neq 0$, as otherwise \eqref{str1} is still verified. The lower bound in $\eqref{str1}$ results from the interlacing property of singular values. With this factorization, the singular values of $R_{11}$ and $R_{22}$ are good approximations of the largest $k$ and the remaining $\sM-k$ singular values of $A$ respectively.
\begin{lemma}
\label{theorem 2.3}
[Lemma 3.1 in \cite{gu1996efficient}]
\label{strong rrqr}
Let $A\in \mathbb{R}^{\sM\times n}$ and $1\leq k\leq \sM$. For a given parameter $f>1$, there exists a permutation $\Pi$ such that
$$A\Pi=Q\begin{pmatrix}
    R_{11} & R_{12}\\
     & R_{22}
    \end{pmatrix},$$
    where $R_{11}\in \mathbb{R}^{k\times k}$ and 
    \begin{equation}
    \label{cond}
        (R_{11}^{-1}R_{12})_{i,j}^{2}+\omega_{i}^{2}(R_{11})\gamma_{j}^{2}(R_{22})\leq f^{2}
    \end{equation}
\end{lemma}
The above permutation is constructed using Algorithm 4 in \cite{gu1996efficient}.
It is shown in [Theorem 3.2 in \cite{gu1996efficient}] that constructing a pivoting strategy satisfying the above property \eqref{cond} is sufficient to achieve a strong rank revealing QR satisfying \eqref{str1} with $\rho_{1}(k,n)=\sqrt{1+f^{2}k(n-k)}$ and \eqref{str2} with $\rho_{2}(k,n)=f$. 
To analyze our algorithm, we use a more relaxed version of Lemma \ref{strong rrqr} shown in \cite{demmel2015communication}.
\begin{corollary}
\label{corollary 2.5}
[Corollary 2.3 in \cite{demmel2015communication}]
Let $A \in \mathbb{R}^{\sM \times n}$ and $1 \leq k \leq \sM$. For a given parameter $f > 1$, there exists a permutation $\Pi$ such that
$$
A\Pi = Q \begin{pmatrix}
    R_{11} & R_{12} \\
    0 & R_{22}
\end{pmatrix},
$$
where $R_{11} \in \mathbb{R}^{k \times k}$ and
\begin{equation}
\label{equation_srrqr}
    \gamma_{j}^{2}(R_{11}^{-1}R_{12})+(\gamma_{j}(R_{22})/\sigma_{min}(R_{11}))^{2}\leq f^{2}k \text{  for  } j=1,...,n-k. 
\end{equation}
\end{corollary}
The proof is trivial, as the permutation $\Pi$ of Theorem \ref{theorem 2.3} verifies \eqref{equation_srrqr} by summing over the index $i$. From Corollary \ref{corollary 2.5}, it is shown, in [Theorem 2.4 in \cite{demmel2015communication}], that if a permutation $\Pi$ satisfies \eqref{equation_srrqr} then it satisfies \eqref{str1} with $\rho_{1}(k,n)=\sqrt{1+F^{2}(n-k)}$.
\subsection{Subspace Embedding}
We present the $\epsilon-$embedding property and give examples of different dense and sparse oblivious and non-oblivious subspace embeddings. 
\begin{definition}[$\epsilon$-Johnson-Lindenstrauss Transform \cite{johnson1984extensions,larsen2017optimality,dasgupta2003elementary}]
Given $\epsilon>0$, a random matrix $\Omega\in \mathbb{R}^{l\times n}$ is an $\epsilon-$Johnson-Lindenstrauss Transform ($\epsilon-$JLT) for a set of vectors $x_{1},...,x_{\sM}$ with $x_{i}\in \mathbb{R}^{n}$, if
\begin{equation}
\label{JL transform}
|\langle x_{i},x_{j}\rangle-\langle\Omega x_{i},\Omega x_{j}\rangle |\leq \epsilon\|x_{i}\|\|x_{j}\|, \quad \forall x_{i},x_{j}.
\end{equation}
\end{definition}
The Johnson-Lindenstrauss Transform (JLT) can be extended to $\epsilon-$subspace embedding where \eqref{JL transform} is true for any two vectors in a $\sM-$dimensional subspace. For classical dense JLT, this extension can be achieved by using popular techniques in stochastic analysis such as $\epsilon-$net and chaining.
\begin{definition}[$\epsilon$-subspace embedding \cite{woodruff2014sketching}]
A sketching matrix $\Omega \in \mathbb{R}^{l \times n}$ is an $\epsilon$-subspace embedding for a vector subspace $\mathcal{W} \subset \mathbb{R}^{n}$, with $\epsilon \in (0,1)$, if
    \begin{equation}
    \label{epsilon_embd}
        |\langle x,y\rangle-\langle\Omega x,\Omega y\rangle |\leq \epsilon\|x\|\|y\|, \quad \forall x,y\in \mathcal{W}.
    \end{equation}
\end{definition}
For $x=y$, we obtain from \eqref{epsilon_embd} that
\begin{equation}
\label{epsilon_embed_2}
        (1-\epsilon)\|x\|_{2}^{2} \leq \|\Omega x\|_{2}^{2} \leq (1+\epsilon)\|x\|_{2}^{2} \quad \forall x \in \mathcal{W}.
    \end{equation}
The $\epsilon\text{-subspace}$ embedding property ensures that the inner product of pairs of vectors within the subspace $\mathcal{W}$ is preserved through sketching, up to a factor of $1 \pm \epsilon$. The following corollary relates the singular values of a matrix $U\in \mathbb{R}^{n\times \sM}$ and the singular values of the sketched matrix $\Omega U\in \mathbb{R}^{l \times \sM}$.
\begin{corollary}[e.g.  \cite{grigori2025randomized}]
\label{bound_singular_values}
    Let $\Omega\in \mathbb{R}^{l\times n}$ be an $\epsilon$-embedding property of a subspace $\mathcal{W}\subset \mathbb{R}^{n}$ and $U\in \mathbb{R}^{n \times \sM}$ be a matrix with $range(U)\subseteq \mathcal{W}$. Then
    \begin{equation}
        \sqrt{(1-\epsilon)}\sigma_{i}(U)\leq \sigma_{i}(\Omega U)\leq \sqrt{(1+\epsilon)} \sigma_{i}(U)
    \end{equation}
for all $1\leq i \leq \sM$.
\end{corollary}
Two types of subspace embeddings can be considered. The first is the non-oblivious embedding, where the sketching matrix is designed to satisfy the $\epsilon\text{-subspace}$ embedding property for a specific subspace $\mathcal{W}$. However, when the subspace is unknown in advance, an oblivious subspace embedding $\Omega$ is constructed such that it verifies, with high probability, the $\epsilon$-embedding property for any fixed subspace $\mathcal{W}\subset \mathbb{R}^{n}$.
\begin{definition}[Oblivious subspace embedding \cite{woodruff2014sketching}]\label{epsilon_sub_embd}
    Let $\epsilon,\delta\in (0,1)$. The sketching matrix $\Omega \in \mathbb{R}^{l \times n}$ is an oblivious subspace embedding with parameters $(\epsilon, \delta, \sM)$ if it satisfies the \(\epsilon\)-subspace embedding property for any d-dimensional subspace $\mathcal{W}\subset \mathbb{R}^{n}$ with probability at least $1-\delta$.
\end{definition}

One of the possible distributions for constructing a sketching matrix $\Omega \in \mathbb{R}^{l \times n}$ that is an oblivious subspace embedding is the Gaussian distribution. Each entry in $\Omega$ is a Gaussian random variable with mean $0$ and variance $\frac{1}{l}$. This embedding benefits from achieving the optimal sketching dimension. Specifically, as shown in \cite{woodruff2014sketching}, it is an oblivious subspace embedding with parameters $(\epsilon, \delta, \sM)$ for a sketching dimension $l=O(\epsilon^{-2}(\sM+\log(\frac{1}{\delta})))$. However, in this case $\Omega$ is a dense matrix, and thus it is expensive to apply it to a vector. Another embedding is the subsampled randomized Fourier transform (SRFT). This sketching matrix has the form  $\Omega = \sqrt{\frac{n}{l}} \, R \, F \, D,$ where $D \in \mathbb{R}^{n \times n}$ is a diagonal matrix whose entries are randomly chosen from $\{+1,-1\}$, $F \in \mathbb{R}^{n \times n}$ is a fast trigonometric transform, and $R \in \mathbb{R}^{l \times n}$ selects $l$ random subset of rows uniformly without replacement. The advantage of such an embedding is the low memory storage. In addition, it is only given as a matrix-vector product, and its application to a vector costs $O(n\log(n))$ flops. In \cite{tropp2011improved} it is shown that this distribution is an  $(\epsilon,\delta,\sM)$-oblivious subspace embedding if the sketch dimension is $l = \mathcal{O}(\epsilon^{-2}(\sM+\log(\frac{n}{\delta})\log(\frac{\sM}{\delta}))$.
\subsubsection{Oblivious Sparse Embeddings}
\label{section 2.3.1}
In this paper we consider first sparse embeddings that have few nonzero entries in each column and are $(\epsilon,\delta,\sM)$-oblivious subspace embeddings. An example of such an embedding is the oblivious sparse norm-approximating projection (OSNAP) introduced in \cite{nelson2013osnap}.
\begin{definition} [OSNAP in \cite{nelson2013osnap}]
Let $s\geq 1$ be an integer and $\Omega\in \mathbb{R}^{l \times n}$ be a random matrix such that $\Omega_{i,j}=\delta_{i,j}g_{i,j}$, where $\delta_{i,j}$ is an indicator random variable for the event $\Omega_{ij}\neq 0$ and $g$ takes values $\{\frac{1}{\sqrt{s}},\frac{-1}{\sqrt{s}}\}$ uniformly at random. $\Omega$ is said to be OSNAP if the following two properties are satisfied:
\begin{itemize}
    \item For any column $j\in [n]$, $\sum_{i=1}^{l}\delta_{i,j}=s$ with probability 1.
    \item For any $S\subseteq[l]\times[n], \mathbb{E}\prod_{(i,j)\in S}\delta_{i,j}\leq (s/l)^{|S|}$.
\end{itemize}
\end{definition}
The first property of OSNAP ensures that $\Omega$ has exactly s nonzero entries in each column and the second property ensures that $\delta_{i,j}$ are negatively correlated. Two approaches to construct an OSNAP distribution are introduced in \cite{kane2014sparser}. The first approach consists of choosing uniformly at random $s$ nonzero entries for each column $j$ in $\Omega$. The values assigned to these nonzero entries are also chosen uniformly at random from $\{-1/\sqrt{s},1/\sqrt{s}\}$. The second approach uses a block-wise partitioning where $l$ rows are divided into $s$ blocks of size $l/s$ and chooses for each column a nonzero entry in each block. In other words, this approach relies on a hash function $h:[n]\times [s] \to [l/s]$, where $h_{i,j}$ is chosen uniformly at random from $[l/s]$, and a random sign function $g:[n]\times [s]\to \{\frac{-1}{\sqrt{s}},\frac{1}{\sqrt{s}}\}$, such that for each $(i,j)\in [n]\times [s]$, $\Omega_{(j-1)s+h(i,j),i}=g_{i,j}$.
\begin{definition}
    An OSNAP matrix $\Omega \in \mathbb{R}^{l \times n}$ with $s=1$ is said  to be a CountSketch matrix.
\end{definition}
The above CountSketch matrix is defined in \cite{clarkson2017low} and it is nothing other than the CountSketch matrix in the data stream literature \cite{charikar2004finding,thorup2012tabulation}.  It is shown in \cite{nelson2013osnap,nelson2014lower} that there is a trade-off between the embedding dimension $l$ and the sparsity parameter $s$ that allows to achieve the $(\epsilon,\delta,\sM)-$oblivious subspace embedding property. A lower sparsity $s$ implies a higher $l$, and as $s$ increases, the embedding dimension $l$ decreases. The following two theorems present the sufficient conditions on $s$ and $l$ which reflects this trade-off.
\begin{theorem}[Theorem 3 in \cite{nelson2013osnap}]
\label{embed}
  Let $ \Omega \in \mathbb{R}^{l \times n}$ be a CountSketch matrix. If the embedding dimension is
\( l = O(d^{2}/\delta\epsilon^{2}) \), then $\Omega$ is an $(\epsilon,\delta,\sM)$-oblivious subspace embedding.
\end{theorem}

\begin{theorem}[Optimal OSNAP \cite{chenakkod2026optimal}]
Let $\Omega \in \mathbb{R}^{l \times n}$ be an OSNAP distribution. For $\epsilon \ge d^{-O(1)}$, if the embedding dimension is $ l = \tilde{O}(d/\epsilon^2) $
and the column sparsity is
$ s = \tilde{O}(\log(d)/\epsilon), $
then $\Omega$ is an $(\epsilon, \delta, d)$-oblivious subspace embedding with high probability $1 - d^{-O(1)}$, where $\tilde{O}(\cdot)$ hides sub-polylogarithmic factors in $d$.
\end{theorem}

\subsubsection{Non-Oblivious Sparse Embeddings}
In the second part of the paper, we discuss the possibility of using non-oblivious sparse embeddings rather than the oblivious one. In this case, the embedding matrix $\Omega$ is constructed such that it gives more importance to certain rows of the matrix $A$. The construction of such $\Omega$ depends on the leverage scores of the input matrix. This approach is referred to in \cite{chenakkod2024optimal} as Leverage Score Sparsification (LESS).

For a general matrix $T\in\mathbb{R}^{n\times \sM}$ with $n\gg \sM$, the statistical leverage scores are defined as the squared row norm of $U$, where $U$ is an orthonormal basis of the column space of $T$, that is, $l_{i}=\|U_{i}\|_{2}^{2}$.
When the matrix is orthogonal, computing $l_{i}$ is inexpensive; however, for general matrices, the leverage scores can be approximated. One common approach \cite{drineas2012fast} consists of first sketching the matrix $T$ from left with an $\epsilon-$ Fast JLT such as SRFT  $\Omega_{1}\in \mathbb{R}^{r_{1}\times n}$, and then computing the SVD of the sketched matrix $T_{sk}=\Omega_{1}T=U_{sk}\Sigma_{sk}V_{sk}^{T}$. This yields a matrix $TV_{sk}\Sigma^{-1}_{sk}$ with $n$ normalized rows. By further sketching  from the right with an $\epsilon-$JLT for $n^{2}$ vectors (the $n$ rows and their pairwise sums), $\Omega_{2}\in \mathbb{R}^{\sM \times r_{2}}$, the approximate leverage scores are 
$$\tilde{l}_{i}=\|\left(TV_{sk}\Sigma^{-1}_{sk}\Omega_{2}\right)_{i,:}\|_{2}^{2} \quad \text{ for all } i\in [n].$$
Two variants of the Leverage Score Sparsification are presented in \cite{chenakkod2024optimal}: one with independent entries and the other one with independent rows. In the first variant, called LESS-IND-ENT, each entry $\Omega_{ij}$ is selected independently with probability proportional to the leverage score of the corresponding component $l_j$ and scaled appropriately. In the second variant, the matrix $\Omega$ is constructed by forming each row $i$ as the sum of several i.i.d matrices $Z_{ij}$, where each $Z_{ij}$ has one nonzero entry in row $i$ chosen with probability proportional to leverage scores $l_{j}$.

The next theorem gives the condition on the embedding dimension such that both constructions form an embedding for the range of a fixed matrix $T$. Moreover, such constructions bound the number of nonzero entries in each row with high probability and show its dependence on the smaller dimension $d$. 
   \begin{theorem}[Theorem 4.3 in \cite{chenakkod2024optimal}]
   \label{leverage_embedding}
   Let $T\in \mathbb{R}^{n\times \sM}$be any fixed matrix with $(\beta_{1},\beta_{2})-$approximate leverage scores $l_{i}$. For any $\epsilon,\delta\in (0,1)$, there exist constant $c_{1},c_{2}$ such that if $l=c_{1}\max\left(d,\log(4/\delta)\right)/\epsilon^{2}$ and $s=c_{2}\log^
   {4}(d/\delta)/ \epsilon^{6}$, then LESS-IND-ENT and LESS-IND-ROWS $\Omega$ are an $(\epsilon,\delta,\sM)-$embedding to $range(T)$ and the maximum number of nonzero entries per row is $O(\log^{4}(d/\delta)/\epsilon^{4})$ with probability at least $1-\delta$.
   \end{theorem}

%% file: RandWSQR_Algorithm.tex
\section{Randomized QR using sparse embeddings}

In this section we present our randomized algorithm, \RandWSQR{}, for selecting $k$ columns from a matrix $A\in \mathbb{R}^{d\times n}$, $d \ll n$, that relies on the QR factorization with column pivoting and sparse embeddings.  We then show that the obtained factorization satisfies the properties of a strong rank-revealing factorization. Let $r$ denote the rank of the matrix $A$.

\subsection{\RandWSQR{} factorization}
Let $\epsilon, \delta \in (0,1)$, $1\leq k\leq \sM$ and $\Omega \in \mathbb{R}^{\ell \times n}$ be a sparse embedding with sparsity $s$ for the subspace $\operatorname{range}(A^{T})$. To begin, we sketch the rows of $A$ to form $B = A\Omega^T$. We then perform a strong rank-revealing QR factorization of $B$ with $\kk\geq k$ to obtain the factorization
\begin{equation}
B \Pi^{B} = Q^B R^B = Q^B \begin{pmatrix}
R_{11}^{B} & R_{12}^{B} \\
0 & R_{22}^{B}
\end{pmatrix}
\end{equation}
with $R_{11}^{B} \in \mathbb{R}^{\kk \times \kk}, R_{12}^{B} \in \mathbb{R}^{\kk\times (l-\kk)}$, and $R_{22}^{B}\in \mathbb{R}^{(\sM-\kk)\times (l-\kk)}$.

To illustrate the column selection process more clearly, we consider that $\Omega$ is a CountSketch matrix. The same steps apply to the general case of a sparse embedding with $\sss>1$. Each column $b_i$ in $B\Pi^{B}$ satisfies $A\Omega^{T}\Pi^{B}\mathbf{e_{i}}=b_{i}$. Let $i_{1},i_{2},...,i_{n_{i}}$ be the nonzero indices of $\Omega^{T}\Pi^{B}\mathbf{e}_{i}$, then $a_{i_{1}}, \ldots, a_{i_{n_{i}}}$ are the columns of $A$ forming $b_i$. Let $\Pi^{A}$ be the permutation matrix that permutes the colums of $A$ such that 
\begin{equation}
A\Pi^{A} = \begin{pmatrix} \tilde{A}_{1} & \tilde{A}_{2}
\end{pmatrix}
\end{equation}
where
$$
\tilde{A}_{1} \in \mathbb{R}^{\sM \times p} = 
\begin{pmatrix}
a_{1_{1}} & a_{1_{2}} & \dots & a_{1_{n_1}} & \dots & a_{\kk_1} & \dots & a_{\kk_{n_k}}
\end{pmatrix}
$$
is the induced column set formed by concatenating those columns, and 
\[ \tilde{A}_{2} \in \mathbb{R}^{\sM \times (n-p)} = \begin{pmatrix}a_{(\kk+1)_{1}}& a_{(\kk+1)_{2}}& \dots& a_{(\kk+1)_{n_{(\kk+1)}}} \ldots& a_{l_{1}}&\dots a_{l_{n_{l}}}\end{pmatrix} \] is the concatenation of the remaining ones. This procedure is shown in Figure \ref{figure_algorithm}.
\begin{figure}[h]
    \centering
    \begin{minipage}{0.4\textwidth}
        \centering
        \includegraphics[width=\linewidth]{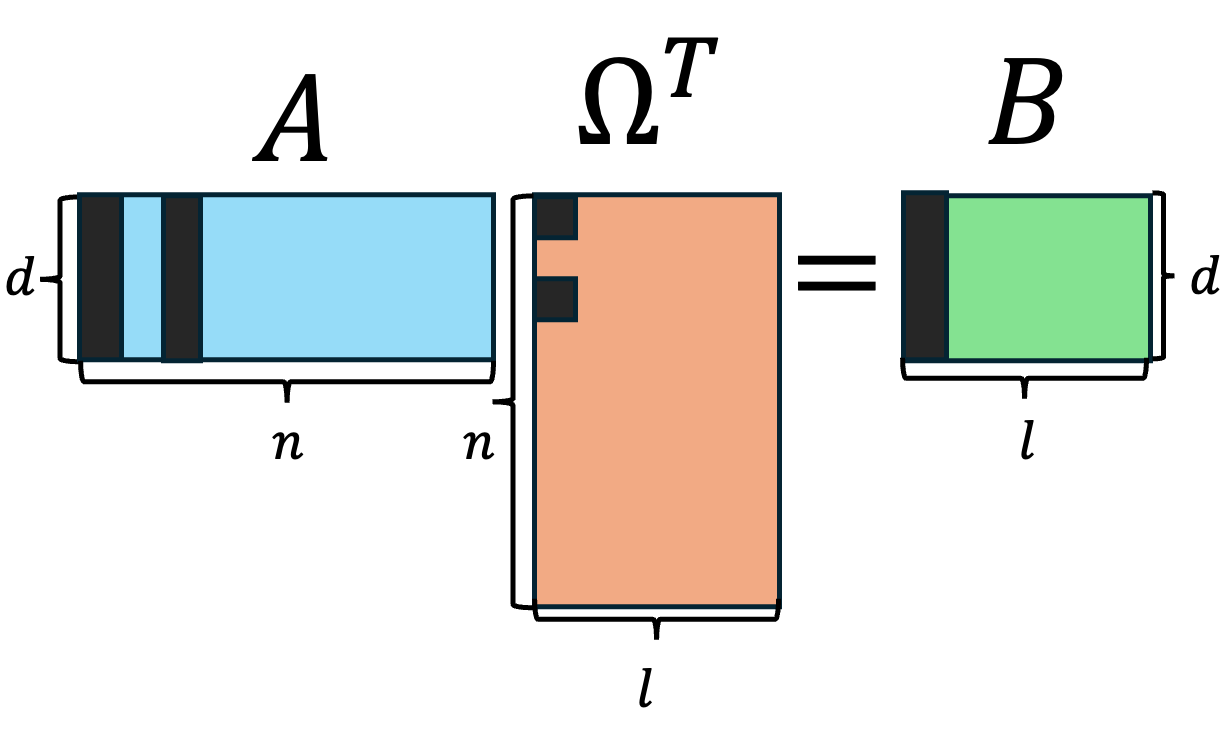}
        \caption*{\textbf{Step 1:} Sketching $A$}
    \end{minipage}
    \hfill
    \begin{minipage}{0.5\textwidth}
        \centering
        \includegraphics[width=\linewidth]{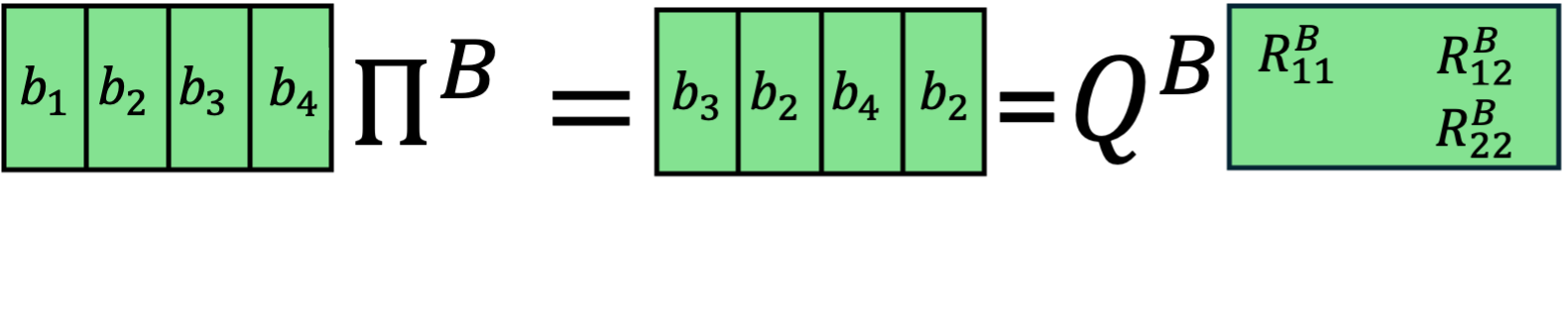}
        \caption*{\textbf{Step 2:} Applying sRRQR on sketched matrix $B$}
    \end{minipage}

    \vspace{0.5cm}  

    \begin{minipage}{0.5\textwidth}
        \centering
        \includegraphics[width=\linewidth]{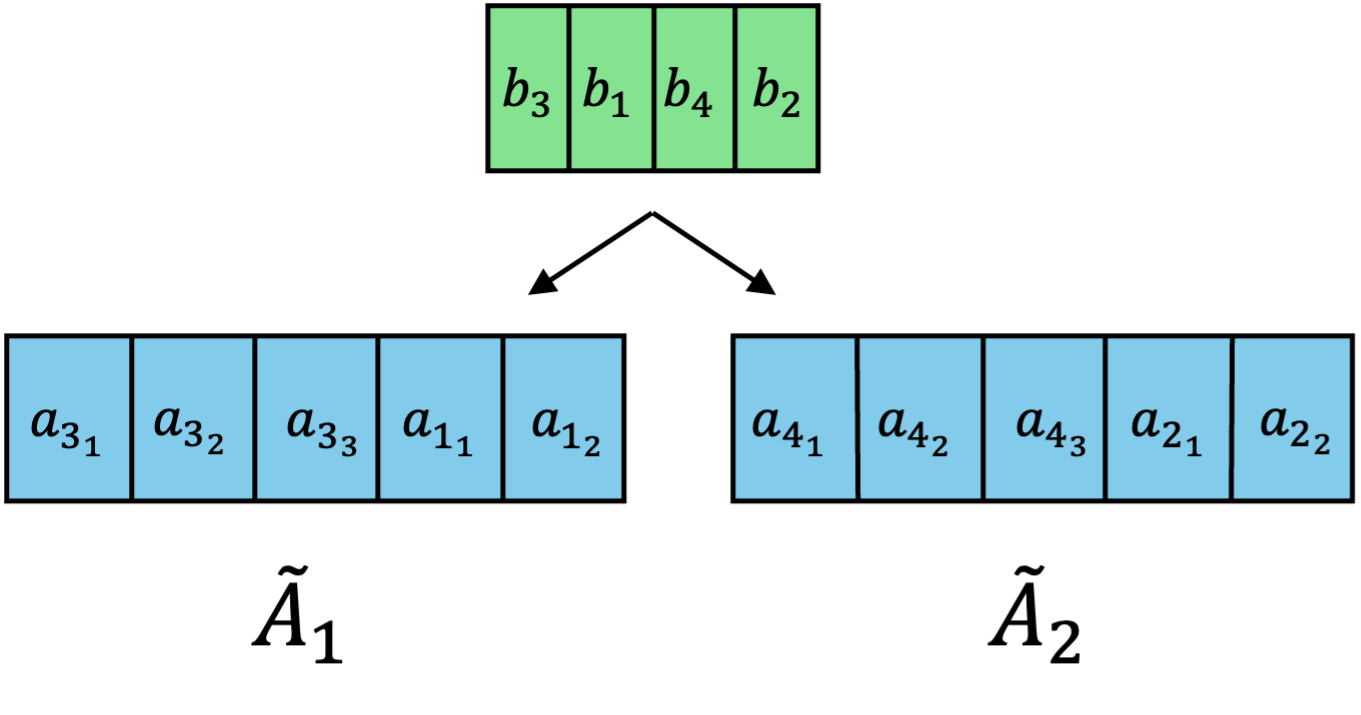}
        \caption*{\textbf{Step 3:} Forming the matrices $\tilde{A}_{1}$ and $\tilde{A}_{2}$}
    \end{minipage}
\caption{Steps of the \RandWSQR{} algorithm}
\label{figure_algorithm}
\end{figure}
After that sRRQR is applied to $\tilde{A}_{1}$ yielding the factors $\tilde{\Pi}, \tilde{Q}, \tilde{R}$ be for the chosen rank $k$, that is
\begin{equation}
\tilde{A}_{1}\tilde{\Pi} = \tilde{Q} \begin{pmatrix}
\tilde{R}_{11} & \tilde{R}_{12} \\
0 & \tilde{R}_{22}
\end{pmatrix},
\end{equation}
with $\tilde{Q}\in \mathbb{R}^{\sM \times \sM}$, $\tilde{R}_{11} \in \mathbb{R}^{k \times k}, \tilde{R}_{12}\in \mathbb{R}^{k\times (p-k)}$, and $\tilde{R}_{22}\in \mathbb{R}^{(\sM -k)\times (p-k)}$. 
Define then 
\[
\Pi = \Pi^{A} \begin{pmatrix}
\tilde{\Pi} & 0 \\
0 & I
\end{pmatrix},
\]
and $Q = \tilde{Q}$, $R_{11} = \tilde{R}_{11}$, 
\[
R_{12} = [\tilde{R}_{12}, [\tilde{Q}^{T}\tilde{A}_{2}](1:k,:)],
\]
\[
R_{22} = [\tilde{R}_{22}, [\tilde{Q}^{T}\tilde{A}_{2}](k+1:\sM,:)],
\]
where
\begin{equation}
\label{The_factorization}
A \Pi = Q \begin{pmatrix}
R_{11} & R_{12} \\
0 & R_{22}
\end{pmatrix}.
\end{equation}
\begin{remark}
    It is sufficient to choose the rank $\kk$ as $k$. When using OSE with sparsity $s>1$ and embedding dimension $\ell$ small, numerically $k'$ can still be chosen slightly smaller that $k$ to prevent the large size $p$. 
\end{remark}
 One of the advantages of using $\Omega$ is the fast application to the matrix $A$ and the low memory storage. It can also be seen that after sketching $A$ from the right with $\Omega^{T}$, each column of the sketched matrix $B$ is a combination of a few columns of the matrix $A$, unlike other dense distributions where each column is a combination of all the columns of the matrix $A$. This property is an essential part for our proposed method. 

Let $w_i$ be a random variable corresponding to the number of nonzeros in row $i$ of $\Omega$. To estimate $\mathbb{E}(w_{i})$ in the oblivious case, for example taking $\Omega$ to be an OSNAP distribution, an analogy can be drawn between the OSNAP distribution and the balls and bins problem, as discussed in \cite{nelson2014lower}.  With this analogy, the nonzero indices can be seen as balls and the rows of $\Omega$ as bins. In the case of CountSketch, for example, $\Omega$ is constructed by throwing $n$ balls uniformly at random into $\ell$ bins. This means that $w_i$ is nothing more than the number of balls in the same bin $i$, which gives $\mathbb{E}(w)=\frac{n}{\ell}$. In the case $n> \ell(\log(\ell))$, we have that the maximum load is $\omega^{*}=\frac{n}{\ell}+\Theta\left(\sqrt{\frac{n\log(\ell)}{\ell}}\right)$ with high probability \cite{raab1998balls}. 

For $s>1$, we are interested in the number of columns that have at least one nonzero entry in a block of $k$ rows of $\Omega$. This represents the number of columns of $A$ involved in forming the corresponding $k$ columns of the sketched matrix $B$.
\begin{proposition}
\label{proposition 2.15}
Let $X$ be the random variable corresponding to the number of nonzero entries in different columns in a block of $k$ rows of $\Omega$. If the sparsity parameter $s\ll \ell$ then $\mathbb{E}(X)=n\left[1-\left(1-\frac{k}{\ell}\right)^{s}\right]$
\end{proposition}
\begin{proof}
    Note that $X$ is the number of columns chosen such that these columns have at least one nonzero entry in the block of $k$ rows. For that, we do not just multiply the number of nonzeros in each row by $k$, since this could count the same column more than once.
    
    For each $j\in [n]$, let $E_{j}$ denote the event of having at least 1 nonzero entry in the $k$ rows of column $j$ and $X_{j}$ be an indicator random variable for the event $E_{j}$, i.e. 
    $$X_{j} = \begin{cases}
        1 & \text{if    } E_{j}\neq \varnothing\\
        0 & \text{if    }  E_{j}=\varnothing
    \end{cases} $$
    Then $X=\sum_{j}X_{j}$.This gives
    \begin{align*}
    \mathbb{E}(X) &= n\mathbb{E}(X_{j})= n\left(1-\frac{\binom{\ell-k}{s}}{\binom{\ell}{s}}\right)\\
    \intertext{For $s\ll \ell$, hypergeometric distribution can be approximated by binomial distribution which gives}
   \mathbb{E}(X) &\approx n\left[1-\left(1-\frac{k}{\ell}\right)^{s}\right]
    \end{align*}
\end{proof}
This guarantees that $\tilde{A}_{1}$ has much fewer columns than $A$, making it more efficient to apply strong rank revealing QR to $\tilde{A}_{1}$ rather than $A$.
\begin{remark}
For the case with $\sss>1$, although the embedding dimension $\ell$ can be much smaller than that of the CountSketch, the size of $\tilde{A}_{1}$, denoted by $p$, is larger. This increase is due to the higher number of nonzero entries in each column, which results in each pivot of the sketched matrix $B$ becoming a linear combination of a larger number of columns, leading to a higher dimension of $\tilde{A}_{1}$. However, in practice, this can be solved by slightly increasing $\ell$ or applying sRRQR to $B$ with rank less than $k$ while still ensuring that $p>\sM$.
\end{remark} 
 We note that we have 
    $$A\Pi^{A} (\Pi^{A})^{T}\Omega^{T} \Pi^{B} = B\Pi^{B}=Q\begin{pmatrix}
        R_{11}^{B}& R_{12}^{B}\\
        &R_{22}^{B}
    \end{pmatrix}.$$
\label{remark omega}
Let us define
\[
\bar{\Omega} = \begin{pmatrix}
    \bar{\Omega}_{11} & \bar{\Omega}_{12} \\
    0_{(n-p)\times (\ell-\kk)}  & \bar{\Omega}_{22}
\end{pmatrix} = (\Pi^{A})^{T} \Omega^{T} \Pi^{B},
\]
where:
\[
\bar{\Omega}_{11} \in \mathbb{R}^{p \times \kk}, \quad \bar{\Omega}_{22} \in \mathbb{R}^{(n-p) \times (\ell - \kk)}, \quad \bar{\Omega}_{12} \in \mathbb{R}^{p\times(\ell-\kk)}.
\]

This column selection procedure and the associated \RandWSQR{} factorization is presented in Algorithm \ref{algorithm1}. First,  a sparse embedding is applied to the matrix $A$ (line 2). Then sRRQR is applied to the sketched matrix $B$ in line 3. By extracting the corresponding indices (line 6), $\tilde{A}_{1}$ is formed (line 7) followed by sRRQR to get the final column selection as in line 8.

\begin{algorithm}
\caption{\RandWSQR{}}
\label{algorithm1}
\begin{algorithmic}

\State \textbf{Input:} Matrix $A\in \mathbb{R}^{\sM\times n}$, oblivious sparse embedding $\Omega\in \mathbb{R}^{\ell\times n}$, ranks $k$ and $\kk$, sRRQR constant $f$

\State \textbf{Output:} Orthogonal matrix $Q\in \mathbb{R}^{\sM \times k}$, upper triangular matrix $R\in \mathbb{R}^{k\times n}$, permutation vector $\Pi$

\State $B = A \times \Omega^T$ 
\Comment{Form the sketched matrix $B$}

\State $[Q^{B}, R^{B}, \Pi^{B}] = \text{sRRQR}(B, f, \kk)$ 
\Comment{Compute sRRQR on $B$}

\State $selected\_pivots = \Pi^{B}(1:\kk)$

\State $selected\_rows = \Omega(selected\_pivots, :)$

\State $[\sim, indices] = \text{find}(selected\_rows)$ 
\Comment{Indices forming first $k$ columns of $B$}

\State $\tilde{A}_{1} = A(:, indices)$ 
\Comment{Form the matrix $\tilde{A}_{1}$}

\State $[\tilde{Q}, \tilde{R}, \tilde{\Pi}] = \text{sRRQR}(\tilde{A}_{1}, f, k)$ 
\Comment{Compute sRRQR on $\tilde{A}_{1}$}

\State $\tilde{A}_{2} = A(:, remaining\_indices)$ 
\Comment{Form the matrix $\tilde{A}_{2}$}

\State $Q = \tilde{Q}(:, 1:k)$

\State $R = \begin{pmatrix} \tilde{R}(1:k,:) & Q^{T} \tilde{A}_{2} \end{pmatrix}$ 
\Comment{Truncate $Q$ and form $R$}

\State $\Pi = [indices*\tilde{\Pi}, remaining\_indices]$ 
\Comment{Permutation vector}

\State \Return $Q, R, \Pi$

\end{algorithmic}
\end{algorithm}

\subsection{Arithmetic Complexity of \RandWSQR{}}
\label{section 3.2}
\RandWSQR{} consists of two parts. In the first part, strong rank revealing QR is applied to the sketched matrix $B$ with $\ell$ columns where $\ell\ll n$. In the second part, strong rank revealing QR is applied to the reduced column matrix $\tilde{A}_{1}$ with $p$ columns. Proposition \ref{proposition 2.15} shows that $\mathbb{E}(p)=n\left[1-\left(1-\frac{\kk}{\ell}\right)^{s}\right]$, which is also much smaller than $n$. This implies that applying sRRQR to $B$ and $\tilde{A}_{1}$ has a lower computational cost than applying it to $A$ with $n$ columns. 

Thus, the computational cost of obtaining the pivots through this algorithm consists first of sketching $A$ with $\Omega^{T}$, which is accelerated due to the use of a sparse matrix, resulting in $O(\sss \sM n)$ operations. Secondly, the application of sRRQR to $B$ and $\tilde{A}_{1}$ is of complexity $O(\sM(\ell\kk+pk))$. The other steps involved have significantly lower computational complexity and are therefore considered minor in comparison. The total complexity of computing the pivots through \RandWSQR{} is $O(\sM(\ell\kk+pk+n\sss))$. To obtain the full factorization, the additional cost of computing $\tilde{Q}^{T}\tilde{A}_{2}$, which is a BLAS3 operation, must be considered. If we choose $\epsilon=\frac{1}{2}$ and consider the case of $\sss=O(\log(\sM))$ with $\ell=O(\sM \log(\sM))$, the complexity is:
\[
O\left(4\sM^{2}\log(\sM)\kk+dkp+\sM n \log(\sM)\right) 
\]
with $\mathbb{E}(p)=n\left[1-\left(1-\frac{\kk}{4\sM\log(\sM)}\right)^{2\log(\sM)}\right]$.
Compared to the complexity of deterministic strong RRQR, $O(ndk)$, it can be seen that the first term, $4\sM\sM\log(\sM)\kk$, does not depend on $n$, making it negligible with respect to the last term for large values of $n$. For the second term, $\mathbb{E}(p)$ is smaller than $n$. This can be seen by checking the worst case where $\kk=d$, $dp\kk=\frac{nd^{2}}{2}$. The sketching step yields a constant-factor improvement over QRCP by approximately a factor of $k/s$ for dense matrices. For sparse matrices, the sketching complexity is $O(\mathrm{nnz}(A)\,s)$, providing additional computational savings as the sparsity of $A$ increases. Indeed, as will be seen later in the numerical results section, our randomized approach leads to a significant reduction in execution time with respect to the deterministic strong RRQR.

\section{Strong Rank Revealing Properties of \RandWSQR{} }
In order to show that the factorization obtained by \RandWSQR{} is strong rank-revealing, we prove that the obtained factorization verifies the inequalities \eqref{inequality_intro}

In this section, we define two matrices $\mathcal{N}$ and $\mathcal{C}$ as: $$\mathcal{N}=\begin{pmatrix}
    I_{k\times k} & \tilde{R}_{11}^{-1}\tilde{R}_{12}
\end{pmatrix}\tilde{\Pi}^{T}\overline{\Omega}_{11} \quad \text{and} \quad \mathcal{C}=\begin{pmatrix}
    0_{(d-k)\times k} & \tilde{R}_{22}
\end{pmatrix}\tilde{\Pi}^{T}\overline{\Omega}_{11}.$$
We follow the proof strategy in \cite{demmel2015communication} combined with our formulation using sparse sketching.
\begin{lemma}
\label{form_lemma}
Let $\tilde{Q},Q^{B},\tilde{R},R^{B}$ and $\tilde{A}_{1}$ be the matrices resulting from applying \RandWSQR{} factorization presented in Algorithm \ref{algorithm1} to $A$ using oblivious sparse embedding $\Omega$ for $range(A^{T})$. The form of $\tilde{Q}^{T}\tilde{A}_{2}\overline{\Omega}_{22}$ is
$$\tilde{Q}^{T}\tilde{A_{2}}\overline{\Omega}_{22}=\begin{pmatrix}
    \tilde{R}_{11}\mathcal{N}N+C^{(1)} - \begin{pmatrix}
        \tilde{R}_{11} & \tilde{R}_{12}
    \end{pmatrix}\tilde{\Pi}^{T}\overline{\Omega}_{12} \\
    \mathcal{C}N+C^{(2)} - \begin{pmatrix}
         & \tilde{R}_{22}
    \end{pmatrix}\tilde{\Pi}^{T}\overline{\Omega}_{12} 
\end{pmatrix},$$
with  $N=(R_{11}^{B})^{-1}R_{12}^{B}$ and $\begin{pmatrix}
    C^{(1)}\\
    C^{(2)}
\end{pmatrix}=\tilde{Q}^{T}Q^{B}\begin{pmatrix}
    \\
    R_{22}^{B}
\end{pmatrix}$.
\end{lemma}
\begin{proof}
    First it can be noticed that 
\begin{align*}    \tilde{Q}^{T}\tilde{A_{2}}\overline{\Omega}_{22}&=\tilde{Q}^{T}Q^{B}\begin{pmatrix}
        R^{B}_{12}\\
        R^{B}_{22}
    \end{pmatrix} - \begin{pmatrix}
        \tilde{R}_{11} & \tilde{R}_{12}\\
         & \tilde{R}_{22}
    \end{pmatrix}\tilde{\Pi}^{T}\overline{\Omega}_{12}\\
    &= \tilde{Q}^{T}Q^{B}\begin{pmatrix}
        R_{11}^{B}N\\
        R_{22}^{B}
    \end{pmatrix} - \begin{pmatrix}
        \tilde{R}_{11} & \tilde{R}_{12}\\
         & \tilde{R}_{22}
    \end{pmatrix}\tilde{\Pi}^{T}\overline{\Omega}_{12}.\\
    \intertext{We decompose $Q^{B}\begin{pmatrix}
        R_{12}^{B}\\
        R_{22}^{B}\\
    \end{pmatrix}$ as }
    \tilde{Q}^{T}\tilde{A}_{2}\overline{\Omega}_{22}&= \tilde{Q}^{T}Q^{B}\begin{pmatrix}
        R_{11}^{B}\\
        \\
    \end{pmatrix}N+\tilde{Q}^{T}Q^{B}\begin{pmatrix}
         \\
         R_{22}^{B}
    \end{pmatrix} - \begin{pmatrix}
        \tilde{R}_{11} & \tilde{R}_{12}\\
         & \tilde{R}_{22}
    \end{pmatrix}\tilde{\Pi}^{T}\overline{\Omega}_{12},
\end{align*}
 with \( N = (R_{11}^{B})^{-1} R_{12}^{B} \). We proceed by writing \begin{equation}
Q^{B} \begin{pmatrix}
    R_{11}^{B} \\
    \\
\end{pmatrix} = \tilde{Q}\begin{pmatrix}
    \tilde{R}_{11} & \tilde{R}_{12}\\
     & \tilde{R}_{22}
\end{pmatrix}\tilde{\Pi}^{T}\overline{\Omega}_{11}.
\label{eq:my_labeled_equation}
\end{equation}
This leads to the following equality:
$$\tilde{Q}^{T}Q^{B}\begin{pmatrix}
    R_{11}^{B}\\
    \\
\end{pmatrix}=\begin{pmatrix}
    \tilde{R}_{11} & \tilde{R}_{12}\\
     & \tilde{R}_{22}
\end{pmatrix}\tilde{\Pi}^{T}\overline{\Omega}_{11}.$$
On the other side, we can also write
$$\tilde{Q}^{T}Q^{B}\begin{pmatrix}
     \\
     R_{22}^{B}
\end{pmatrix}=C=\begin{pmatrix}
    C^{(1)}\\
    C^{(2)}
\end{pmatrix}.$$
This allows us to rewrite 
$\tilde{Q}^{T}\tilde{A_{2}}\overline{\Omega}_{22}=\begin{pmatrix}
    \tilde{R}_{11}\mathcal{N}N+C^{(1)} - \begin{pmatrix}
        \tilde{R}_{11} & \tilde{R}_{12}
    \end{pmatrix}\tilde{\Pi}^{T}\overline{\Omega}_{12} \\
    \mathcal{C}N+C^{(2)} - \begin{pmatrix}
        0_{(d-k)\times k} & \tilde{R}_{22}
    \end{pmatrix}\tilde{\Pi}^{T}\overline{\Omega}_{12} 
\end{pmatrix}.$
\end{proof}
\begin{remark}
\label{remark_bound_embedding}
In order to bound 
\[
\|\tilde{A}_{1}\begin{pmatrix} \overline{\Omega}_{11} & \overline{\Omega}_{12} \end{pmatrix}\|_{2}^{2} \leq \|\tilde{A}_{1}\|_{2}^{2} \left\| \begin{pmatrix} \overline{\Omega}_{11} & \overline{\Omega}_{12} \end{pmatrix} \right\|_{1} \left\| \begin{pmatrix} \overline{\Omega}_{11} & \overline{\Omega}_{12} \end{pmatrix} \right\|_{\infty} \leq \omega^{*} \|\tilde{A}_{1}\|_{2}^{2}
\]
with high probability, we leverage the maximum load of allocating $p$ balls into $\ell$ bins across different regimes according to \cite[Theorem 1]{raab1998balls}.
\begin{itemize}
   \item If $\ell / \text{polylog}(\ell) \le sp \ll \ell \log \ell$, then by choosing a constant $\alpha > 1$, we have
    \[
    \omega^{*} = \frac{\log \ell}{\log \left( \frac{\ell \log \ell}{sp} \right)} \left( 1 + \alpha \frac{\log^{(2)} \ell}{\log \ell} \right)
    \]
    such that $\Pr \left[ \left\| \begin{pmatrix} \overline{\Omega}_{11} & \overline{\Omega}_{12} \end{pmatrix} \right\|_{1} > \omega^{*} \right] = o(1)$.
\item If $sp = c \cdot \ell \log \ell$ for some constant $c$, then by choosing a constant $\alpha > 1$, we have
    \[
    \omega^{*} = (d_c^{-1} + \alpha) \log \ell
    \]
    such that $\Pr \left[ \left\| \begin{pmatrix} \overline{\Omega}_{11} & \overline{\Omega}_{12} \end{pmatrix} \right\|_{1} > \omega^{*} \right] = o(1)$, where $d_c$ is a constant depending only on $c$.
\item If $\ell \log \ell \ll sp \le \ell \cdot \text{polylog}(\ell)$, then by choosing a constant $\alpha > 1$, we have
    \[
    \omega^{*} = \frac{sp}{\ell} + \alpha \sqrt{\frac{2sp \log \ell}{\ell}}
    \]
    such that $\Pr \left[ \left\| \begin{pmatrix} \overline{\Omega}_{11} & \overline{\Omega}_{12} \end{pmatrix} \right\|_{1} > \omega^{*} \right] = o(1)$.
   \item If $sp \gg \ell(\log \ell)^{3}$, then by choosing a constant $\alpha > 1$, we have 
    \[
    \omega^{*} = \frac{sp}{\ell} + \sqrt{\frac{2sp\log(\ell)}{\ell}\left(1-\frac{1}{\alpha}\frac{\log^{(2)}(\ell)}{2\log(\ell)}\right)}
    \] 
    such that $\Pr \left[ \left\| \begin{pmatrix} \overline{\Omega}_{11} & \overline{\Omega}_{12} \end{pmatrix} \right\|_{1} > \omega^{*} \right] = o(1)$.
   \item Alternatively, one could require the subspace embedding property to hold uniformly over all choices of submatrices, yielding:
    \begin{equation}
    \label{equation1_bound_uniform}
    \|\tilde{A}_{1}\begin{pmatrix} \overline{\Omega}_{11} & \overline{\Omega}_{12} \end{pmatrix}\|_{2} \leq (1+\epsilon)\|\tilde{A}_{1}\|_{2}.
    \end{equation}
    However, achieving this requires a union bound over all $\binom{n}{p}$ possible column combinations. Ensuring that $\Omega$ acts as an oblivious subspace embedding simultaneously for all combinations introduces a logarithmic dependence on $n$ into both the embedding dimension and the failure probability. This would result in $\omega^{*}=1+\epsilon$ in Lemma \ref{bound_lemma}. However, here we restrict our analysis to the structural conditions outlined previously.
\end{itemize}
\end{remark}
    

\begin{lemma}
\label{bound_lemma}
Let $\tilde{A}_{1}$ and $B$ be the matrices resulting from \RandWSQR{} factorization presented in Algorithm \ref{algorithm1} on $A$ using oblivious sparse embedding $\Omega$ for $range(A^{T})$. Denote $\alpha = \frac{\sigma_{max}(R_{22})}{\sigma_{min}(R_{11})}$. If strong rank revealing QR is applied to $B$ and $\tilde{A_{1}}$ then
$$\sqrt{1+\alpha^{2}+\|\begin{pmatrix}
        I & R_{11}^{-1}R_{12}
    \end{pmatrix}\|_{2}^{2}}\leq \sqrt{ 1+\frac{4\omega^{*}}{1-\epsilon}(1+f^{2}k(p-k)(1+f^{2}\kk(\ell-\kk))}.$$
\end{lemma}
where $\omega^{*}=O(\frac{\sqrt{s}p}{\ell})$
\begin{proof}
    \begin{align*}
 \alpha^{2}+\|\begin{pmatrix}
        I & R_{11}^{-1}R_{12}
    \end{pmatrix}\|_{2}^{2}
    & = \frac{\sigma_{max}^{2}(R_{22})}{\sigma_{min}^{2}(R_{11})}+\|\begin{pmatrix}
        I & R_{11}^{-1}R_{12}
    \end{pmatrix}\|_{2}^{2}\\
    &= \frac{\left\|\begin{pmatrix}
       0 & \tilde{R}_{22} & (\tilde{Q}^{T})^{(2)}\tilde{A}_{2}
    \end{pmatrix}\right\|_{2}^{2}}{\sigma_{min}^{2}(R_{11})}+\left\|\begin{pmatrix}
      I &  R_{11}^{-1}\tilde{R}_{12} & R_{11}^{-1}(\tilde{Q}^{T})^{(1)}\tilde{A}_{2}
    \end{pmatrix}\right\|_{2}^{2}
    \end{align*}
We have that $range(\begin{pmatrix}
       0 & \tilde{R}_{22} & (\tilde{Q}^{T})^{(2)}\tilde{A}_{2}
    \end{pmatrix})$ belongs to $range(A^{T})$, so using $(\epsilon,\delta,r)-$OSE $\Omega$ for $range(A^{T})$ we obtain the following:
\begin{align*}
    \frac{\left\|\begin{pmatrix}
       0 & \tilde{R}_{22} & (\tilde{Q}^{T})^{(2)}\tilde{A}_{2}
    \end{pmatrix}\right\|_{2}^{2}}{\sigma_{min}^{2}(R_{11})}& \leq \frac{1}{1-\epsilon}\left[\frac{\left\|\begin{pmatrix}
       0 & \tilde{R}_{22} & (\tilde{Q}^{T})^{(2)}\tilde{A}_{2}
    \end{pmatrix}\begin{pmatrix}
        \tilde{\Pi}^{T} & \\
         & I
    \end{pmatrix}\overline{\Omega}\right\|_{2}^{2}}{\sigma_{min}^{2}(R_{11})}\right]\\
    &\leq \frac{1}{1-\epsilon}\left[\frac{\left\|\begin{pmatrix}
   \mathcal{C} & \begin{pmatrix}
        0 & \tilde{R}_{22}
    \end{pmatrix}\tilde{\Pi}^{T}\overline{\Omega}_{12}+(\tilde{Q}^{T})^{(2)}\tilde{A}_{2}\overline{\Omega}_{22}\end{pmatrix}\right\|_{2}^{2}}{\sigma_{min}^{2}(R_{11})}\right]\\
    &\leq \frac{1}{1-\epsilon}\left[\frac{\left\|\begin{pmatrix}
        \mathcal{C} & \mathcal{C}N+C^{(2)}
    \end{pmatrix}\right\|_{2}^{2}}{\sigma_{min}^{2}(R_{11})}\right]\\
    &\leq \frac{1}{1-\epsilon}\left[\frac{\left\|\mathcal{C}\right\|_{2}^{2}}{\sigma_{min}^{2}(R_{11})}+\frac{\left\|\mathcal{C}N+C^{(2)}\right\|_{2}^{2}}{\sigma_{min}^{2}(R_{11})}\right].\\
    \intertext{As $range\left(\begin{pmatrix}
        0 & \tilde{R}_{22}
    \end{pmatrix}^{T}\right)$ $\subseteq range(\tilde{A}_{1}^{T})$, using remark \ref{remark_bound_embedding} we obtain:}
    \frac{\left\|\begin{pmatrix}
       0 & \tilde{R}_{22} & (\tilde{Q}^{T})^{(2)}\tilde{A}_{2}
    \end{pmatrix}\right\|_{2}^{2}}{\sigma_{min}^{2}(R_{11})}&\leq \frac{1}{1-\epsilon}\left[\omega^{*}\frac{\|\tilde{R}_{22}\|_{2}^{2}}{\sigma_{min}^{2}(R_{11})}+\frac{2\|\mathcal{C}\|_{2}^{2}\|N\|_{2}^{2}}{\sigma_{min}^{2}(R_{11})}+\frac{2\|C^{(2)}\|_{2}^{2}}{\sigma_{min}^{2}(R_{11})}\right]\tag{*}\label{term_1}.\\
\end{align*}
In this inequality we also used the fact that 
$\|a+b\|_{2}^{2}\leq (\|a\|_{2}^{2}+\|b\|_{2}^{2}+2\|a\|_{2}\|b\|_{2})\leq 2(\|a\|_{2}^{2}+\|b\|_{2}^{2}).$
Similarly we have that $range(\begin{pmatrix}
      I &  R_{11}^{-1}\tilde{R}_{12} & R_{11}^{-1}(\tilde{Q}^{T})^{(1)}\tilde{A}_{2}
    \end{pmatrix})$ belongs to $range(A^{T})$, so using $(\epsilon, \delta, r)-$OSE $\Omega$ for $range(A^{T})$ we obtain the following:
\begin{align*}
    \left\|\begin{pmatrix}
      I &  R_{11}^{-1}\tilde{R}_{12} & R_{11}^{-1}(\tilde{Q}^{T})^{(1)}\tilde{A}_{2}
    \end{pmatrix}\right\|_{2}^{2}&\leq \frac{1}{1-\epsilon}\left[\left\|\begin{pmatrix}
      I &  R_{11}^{-1}\tilde{R}_{12} & R_{11}^{-1}(\tilde{Q}^{T})^{(1)}\tilde{A}_{2}
    \end{pmatrix}\begin{pmatrix}
        \tilde{\Pi}^{T} & \\
         & I
    \end{pmatrix}\overline{\Omega}\right\|_{2}^{2}\right]\\
     &\leq \frac{1}{1-\epsilon}\left[\left\|\begin{pmatrix}
  \mathcal{N} & \begin{pmatrix}
        I & \tilde{R}_{11}^{-1}\tilde{R}_{12}
    \end{pmatrix}\tilde{\Pi}^{T}\overline{\Omega}_{12}+(R_{11}^{-1}(\tilde{Q}^{T})^{(1)}\tilde{A}_{2}\overline{\Omega}_{22}\end{pmatrix}\right\|_{2}^{2}\right]\\
    &\leq \frac{1}{1-\epsilon}\left[\left\|\begin{pmatrix}
        \mathcal{N} & \mathcal{N}N+R_{11}^{-1}C^{(1)}
    \end{pmatrix}\right\|_{2}^{2}\right]\\
    &\leq\frac{1}{1-\epsilon}\left[\left\|\mathcal{N}\right\|_{2}^{2}+ 2\left\|\mathcal{N}N\right\|_{2}^{2}+2\left\|R_{11}^{-1}C^{(1)}\right\|_{2}^{2}\right].\\
    \intertext{As $range(\begin{pmatrix}
        I & \tilde{R}_{11}^{-1}\tilde{R}_{12}
    \end{pmatrix}^{T}) \subseteq$ $range(\tilde{A}_{1}^{T})$, then using remark \ref{remark_bound_embedding} we obtain}
   \left\|\begin{pmatrix}
      I &  R_{11}^{-1}\tilde{R}_{12} & R_{11}^{-1}(\tilde{Q}^{T})^{(1)}\tilde{A}_{2}
    \end{pmatrix}\right\|_{2}^{2} &\leq \frac{1}{1-\epsilon}\left[\omega^{*}\left(\left\|\begin{pmatrix}
        I & \tilde{R}_{11}^{-1}\tilde{R}_{12}
    \end{pmatrix}\right\|_{2}^{2}+2\|\begin{pmatrix}
        I & \tilde{R}_{11}^{-1}\tilde{R}_{12}
        \end{pmatrix}N\|_{2}^{2}\right)+2\|R_{11}^{-1}C^{(1)}\|_{2}^{2}\right].\tag{**}\label{term_2}
\end{align*}
Combining the first terms in inequalities \eqref{term_1} and \eqref{term_2} we obtain
\begin{align*}
\omega^{*}\left(1+\|\tilde{R}_{11}^{-1}\tilde{R}_{12}\|_{2}^{2}+\frac{\|\tilde{R}_{22}\|_{2}^{2}}{\sigma_{min}^{2}(R_{11})}\right)&\leq \omega^{*}\left(1+f^{2}k(p-k)\right).
\end{align*}
For the second terms of \eqref{term_1} and \eqref{term_2} we have
\begin{align*}
    2\omega^{*}\|N\|_{2}^{2}\left(\|\begin{pmatrix}
        I & \tilde{R}_{11}^{-1}\tilde{R}_{12}
    \end{pmatrix}\|_{2}^{2}+\frac{\|\begin{pmatrix}
        0 & \tilde{R}_{22}
    \end{pmatrix}\|_{2}^{2}}{\sigma_{min}^{2}(R_{11})}\right)
    &\leq \omega^{*}\|N\|_{2}^{2}\left(1+f^{2}k(p-k)\right).
\end{align*}
For the third terms, using the fact that $\|C\|_{2}^{2}\leq \|R_{22}^{B}\|_{2}^{2}$ we obtain
\begin{align*}
  \frac{2}{\sigma_{min}^{2}(R_{11})}\left(\|C^{(1)}\|_{2}^{2}+\|C^{(2)}\|_{2}^{2}\right) &\leq \frac{4\|R_{22}^{B}\|_{2}^{2}}{\sigma_{min}^{2}(R_{11})}.
\end{align*}
Furthermore, equation \eqref{eq:my_labeled_equation} can be rewritten as 
$$(R_{11}^{B})^{T}R_{11}^{B}=\mathcal{N}^{T}\tilde{R}_{11}^{T}\tilde{R}_{11}\mathcal{N}+\mathcal{C}^{T}\mathcal{C},$$
which implies that
\begin{align*}
    \sigma_{min}^{2}(R_{11}^{B})&\leq \sigma_{min}^{2}(\tilde{R}_{11})\|\mathcal{N}\|_{2}^{2}+\|\mathcal{C}\|_{2}^{2}\\
    &\leq \omega^{*} \sigma_{min}^{2}(\tilde{R}_{11})\left\|\begin{pmatrix}
        I & \tilde{R}_{11}^{-1}\tilde{R}_{12} \\
         & \tilde{R}_{22}/(\sigma_{min}(\tilde{R}_{11}))
    \end{pmatrix}\right\|_{2}^{2}\leq \omega^{*}\left(1+f^{2}k(p-k)\right)\sigma_{min}^{2}(\tilde{R}_{11}),
\end{align*}
which in turn leads to 
$$1/\sigma_{min}^{2}(R_{11})\leq (1
+\epsilon)\left(1+f^{2}k(p-k)\right)/\sigma_{min}^{2}(R_{11}^{B}).$$
Hence, this leads to
\begin{align*}
    \frac{4\|R_{22}^{B}\|_{2}^{2}}{\sigma_{min}^{2}(R_{11})}\leq \omega^{*}(1+f^{2}k(p-k))\frac{4\|R_{22}^{B}\|_{2}^{2}}{\sigma_{min}^{2}(R_{11}^{B})}
\end{align*}
Combining the three terms together gives
\begin{align*}
\alpha^{2}+\|\begin{pmatrix}
    I & R_{11}^{-1}R_{12}
\end{pmatrix}\|_{2}^{2}&\leq \frac{4\omega^{*}}{1-\epsilon}\left[\left((1+f^{2}k(p-k)\right)\left(1+\|(R_{11}^{B})^{-1}R_{12}^{B}\|_{2}^{2}+\frac{\|R_{22}^{B}\|_{2}^{2}}{\sigma_{min}^{2}(R_{11}^{B})}\right)\right]\\
&\leq \frac{4\omega^{*}}{1-\epsilon}(1+f^{2}k(p-k)(1+f^{2}\kk(\ell-\kk)).
\end{align*}
\end{proof}
\begin{theorem}
    \label{maintheroem}
    Let $\tilde{A}_{1}$ and $B$ be the matrices obtained from \RandWSQR{} on A presented in Algorithm \ref{algorithm1} using oblivious sparse embedding $\Omega$ for $range(A^{T})$. If sRRQR is applied to both matrices then the resulting permutation $\Pi$ satisfies the strong rank revealing QR properties such that 
  \begin{equation}
  \label{result_1}
  1\leq \frac{\sigma_{i}(A)}{\sigma_{i}(R_{11})} ,\frac{\sigma_{j}(R_{22})}{\sigma_{j+k}(A)}\leq \rho_{1}(k,n)
  \end{equation} 
  \begin{equation} 
  \label{result_2}
  \left\| R_{11}^{-1} R_{12} \right\|_{2} \leq \rho_{2}(k,n) 
  \end{equation}
    with $\rho_{1}(k,n)\leq \sqrt{ 1+\frac{4\omega^{*}}{1-\epsilon}(1+f^{2}k(p-k)(1+f^{2}\kk(\ell-\kk))}$ and\\ $\rho_{2}(k,n)\leq \sqrt{ \frac{2\omega^{*}}{1-\epsilon}(1+f^{2}k(p-k)(1+f^{2}\kk(\ell-\kk))}$ and $\omega=O(\frac{\sqrt{s}p}{\ell})$.
\end{theorem}
\begin{proof}
    Here we give the formulation presented in \cite{meng2013low}.
    We start by writing 
    $$A\Pi = Q\begin{pmatrix}
        R_{11} & \\
         & R_{22}/\alpha
    \end{pmatrix}\begin{pmatrix}
        I & R_{11}^{-1}R_{12}\\
         & \alpha I 
    \end{pmatrix}.$$
It follows that for $1\leq i \leq k$, using Weyl's inequality 
 $$\sigma_{i}(A)\leq \sigma_{i}\left( \begin{pmatrix}
        R_{11} & \\
         & R_{22}/\alpha
    \end{pmatrix}\right)\left\|\begin{pmatrix}
        I & R_{11}^{-1}R_{12}\\
         & \alpha I 
    \end{pmatrix} \right\|_{2}.$$
 The largest $k$ singular values of $\begin{pmatrix} R_{11} & 0 \\ 0 & R_{22}/\alpha \end{pmatrix}$ are those of $R_{11}$, and 
    $$\left\|\begin{pmatrix}
        I & R_{11}^{-1}R_{12}\\
         & \alpha I 
    \end{pmatrix} \right\|_{2}\leq \sqrt{\alpha^{2}+\|\begin{pmatrix}
        I & R_{11}^{-1}R_{12}
    \end{pmatrix}\|_{2}^{2}}.$$
    This gives
    $$\sigma_{i}(A)\leq \sqrt{\alpha^{2}+\|\begin{pmatrix}
        I & R_{11}^{-1}R_{12}
    \end{pmatrix}\|_{2}^{2}} \sigma_{i}(R_{11}).$$
On the other side,
$$\begin{pmatrix}
    \alpha R_{11} & \\
     & R_{22}
\end{pmatrix}=\begin{pmatrix}
    R_{11} & R_{12} \\
     & R_{22}
\end{pmatrix}\begin{pmatrix}
    \alpha I & -R_{11}^{-1}R_{12}\\
     & I
\end{pmatrix}.$$
The last $\sM - k$ singular values of $\begin{pmatrix} \alpha R_{11} & 0 \\ 0 & R_{22} \end{pmatrix}$ are those of $R_{22}$.
Hence, for $1\leq j\leq \sM-k$, we have
$$\sigma_{j}(R_{22})\leq \sigma_{j+k}(A)\sqrt{1+\alpha^{2}+\|\begin{pmatrix}
        I & R_{11}^{-1}R_{12}
    \end{pmatrix}\|_{2}^{2}},$$
Combining the above inequalities with \ref{bound_lemma} gives the bound in \eqref{result_1}.  

For the second inequality \eqref{result_2}, one observes that, using \eqref{term_2}
\begin{align*}
    \|R_{11}^{-1}R_{12}\|_{2}^{2}&\leq \left \|\begin{pmatrix}
        I & R_{11}^{-1}R_{12}
    \end{pmatrix}\right\|_{2}^{2}\\
    &\leq \frac{1}{1-\epsilon}\left[\omega^{*}\left(\left\|\begin{pmatrix}
        I & \tilde{R}_{11}^{-1}\tilde{R}_{12}
    \end{pmatrix}\right\|_{2}^{2}+2\|\begin{pmatrix}
        I & \tilde{R}_{11}^{-1}\tilde{R}_{12}
        \end{pmatrix}N\|_{2}^{2}\right)+2\|R_{11}^{-1}C^{(1)}\|_{2}^{2}\right] \\ 
\intertext{Applying a similar argument as in \ref{bound_lemma} completes the proof.}
\end{align*}
\end{proof}

%% file: Leverage_scores_case.tex
\subsection{\RandWSQR{} using non-oblivious sparse embeddings}
In this section, we consider the case where $A$ is an orthogonal matrix or has an approximate leverage scores. We show that in such settings, using a non-oblivious sparse embeddings such as Leverage Score Sparsification (LESS) instead of oblivious one in \RandWSQR{} factorization gives a tighter bound $\rho_{1}(k,n)$ and $\rho_{2}(k,n)$. Indeed, the maximum number of nonzero entries in rows of LESS embedding $\Omega$ by \ref{leverage_embedding} is independent of $n$, the larger dimension, and thus both parameters $p$ and $\ell$ in Theorem \ref{maintheroem} depends only on the smaller dimension $d$.
The next theorem gives the bound for the strong rank-revealing properties of the \RandWSQR{} factorization using LESS embedding. 
   \begin{theorem}
       \label{maintheorem_leverage}
       Let $A\in \mathbb{R}^{d\times n}$ with $d\ll n$ and $\Omega\in \mathbb{R}^{\ell\times n}$ be a non-oblivious LESS embedding for $range(A^{T})$. If $\tilde{A}_{1}$ and $B$ are the matrices obtained from \RandWSQR{} on $A$ presented in Algorithm \ref{algorithm1} and sRRQR is applied to both matrices with constant $f$ then, with high probability, the resulting permutation $\Pi$ verifies inequalities \eqref{result_1} and \eqref{result_2} with
        $$\rho_{1}(k,n)=\rho_{2}=O\left(\frac{k\sqrt{\kk}\log^{4}(d)\left(\frac{d}{\epsilon^{2}}-\kk\right)^{1/2}}{l_{min}^{1/2}\epsilon^{4}(1-\epsilon)^{1/2}}\right)$$
   \end{theorem}
   \begin{proof}
      First we note that as $\Omega$ is a non-oblivious embedding and the nonzero entries are no longer chosen uniformly at random. In this case, remark  \ref{remark_bound_embedding} gives
   $$\|\tilde{A}_{1}\begin{pmatrix}
       \overline{\Omega}_{11} & \overline{\Omega}_{12}
   \end{pmatrix}\|_{2}^{2}\leq \omega^{*}\|\tilde{A}_{1}\|_{2}^{2}$$
   where $\omega^{*}=O\left(\frac{\log^{4}(d/\delta)}{l_{min}\epsilon^{4}}\right)$. By combining this bound and Theorem \ref{maintheroem}, we get, with high probability, 
   $$\rho_{1}(k,n)=\sqrt{1+\frac{4\omega^{*}}{1-\epsilon}(1+f^{2}k(p-k)(1+f^{2}\kk(\ell-k'))}$$
   and 
   $$\rho_{2}(k,n)=\sqrt{\frac{2\omega^{*}}{1-\epsilon}(1+f^{2}k(p-k)(1+f^{2}\kk(\ell-k'))}$$
   Using the bounds of $p$ and $\omega^{*}$ from Theorem \ref{leverage_embedding} we get the desired function $\rho_{1}(k,n)$ and $\rho_{2}(k,n)$ with high probability.
 \end{proof}
We note here that the running time for computing the approximation of leverage scores is $$O(nd\log(d\epsilon^{-1}+nd\epsilon^{-2}\log(n)+d^{3}\epsilon^{-2}(\log(n))(\log(d\epsilon^{-1})))$$
as shown in \cite{drineas2012fast}, which increases the complexity of \RandWSQR{} factorization, however, when the leverage scores are known, the values of $p$ and $\ell$ are both independent of the larger dimension $n$ and gives a greater speedup of the algorithm for $n\gg d$.

%% file: Numerical_results.tex
\section{Numerical Results}
In this section we discuss the efficiency of \RandWSQR{} in terms of approximations of singular values and runtime for various wide and short matrices presented in Table \ref{tab:matrix_list}. 
All experiments were performed in MATLAB (version R2025b) on a Macbook Pro with an Apple M3 processor. The sparse embeddings are generated using the \textit{sparsesign} function, which is implemented in $C$ \cite{epperly2024fast}. The factorizations on the sketched matrix $B$ and reduced column set $\tilde{A}_{1}$ are performed using QR with column pivoting rather than strong RRQR, since in practice sRRQR was not performing any extra column permutation over QRCP. In these tests, we first plot the singular values of $R_{11}$ obtained from both \RandWSQR{} and QRCP with the singular value of $A$, $\sigma_{i}(A)$, obtained from SVD. To further illustrate the comparison, we also display the ratio $\sigma_{i}(R_{11})/\sigma_{i}(A)$ for both methods. In addition, we report the number of columns $p$ selected through the sRRQR of $B$, the sketch of $A$, to form the matrix $\tilde{A}_{2}$. In addition we provide a comparison between using OSE and LESS embedding in \RandWSQR{} in terms of accuracy, runtime and size $P$. Each experiment is repeated $20$ times and the result of the median is shown. For the matrices Fiedler, Chebvand and  Prolate, the matrix $A$ is formed by taking the transpose of the first $n \times \sM$ block.  The matrices used in Subsection \ref{experiments_countsketch} are of dimension $50\times 10000$,in Subsection \ref{experiments_s>1} are of dimension $200\times 10000$ and in subsections \ref{comparison subsection}-\ref{computational time} are of size $100\time 10^{6}$, unless otherwise specified. The rank $k$ is chosen so that the relative spectral norm error $\|R_{22}\|_{2}/\|A\|_{2}$ is sufficiently small and satisfies the values reported in Table \ref{tab:residual error}. In all experiments $k'$ is chosen to be $k$.

\begin{table}[htbp]
\centering
\begin{tabular}{|p{4cm}|p{8cm}|}
\hline
\textbf{Type} & \textbf{Description} \\ 
\hline
Exponential & Random matrix with exponential decaying singular values: $\sigma_{i}=\alpha^{i-1}$ with $\alpha=10^{\frac{-1}{11}}$\\ 
\hline
Quadratic & Random matrix with quadratic decaying singular values. \\ 
\hline
Random & Random Gaussian matrix with entries drawn from a standard normal distribution. \\ 
\hline
ROM & Random Gaussian matrix with $40$ and $100$ random columns replaced by outliers of magnitude $1000$ for matrix of size $50\times 10000$ and $200\times 10000$ respectively.   \\ 
\hline
Low rank matrices  & Random $50\times 10000$ and $200\times
 10000$ matrices of rank $30$ and $100$ respectively. \\ 
\hline
Fiedler & Fiedler symmetric matrix where the entries are defined by the absolute differences between elements. It is ill conditioned.\\ 
\hline
Chebvand & Vandermonde-like matrix for the Chebyshev polynomials. It is ill-conditioned \\ 
\hline
Prolate & Ill-conditioned Toeplitz matrix.\\
\hline
Genomic Dataset & RNA sequence dataset from The Cancer Genome Atlas (TCGA) Breast Invasive Carcinoma.\\ 
\hline
\end{tabular}
\caption{List of matrices used in the experiments.}
\label{tab:matrix_list}
\end{table}
\subsection{\RandWSQR{} with CountSketch oblivious sparse embedding} 

\label{experiments_countsketch}
Figure~\ref{fig:results_s1} shows the comparison between the first $k$ singular values $\sigma_{i}(R_{11})$ obtained from \RandWSQR{} using a CountSketch matrix ($s=1$), deterministic QRCP, and the exact values from SVD across four different matrices: a matrix with exponential decay, a matrix with quadratic decay, a Prolate matrix, and a Chebvand matrix. The sketching dimension $\ell=2500$. 

Figure~\ref{fig:placeholder} presents the corresponding ratios of these singular values, $\sigma_{i}(R_{11})/\sigma_{i}(A)$. The results demonstrate that although the approximation is not close to the SVD baseline, \RandWSQR{} performs similarly to deterministic QRCP and achieves a highly comparable performance.

\begin{figure}[H]
    \centering
    \begin{subfigure}[b]{0.48\textwidth}
        \centering
        \includegraphics[width=\textwidth]{"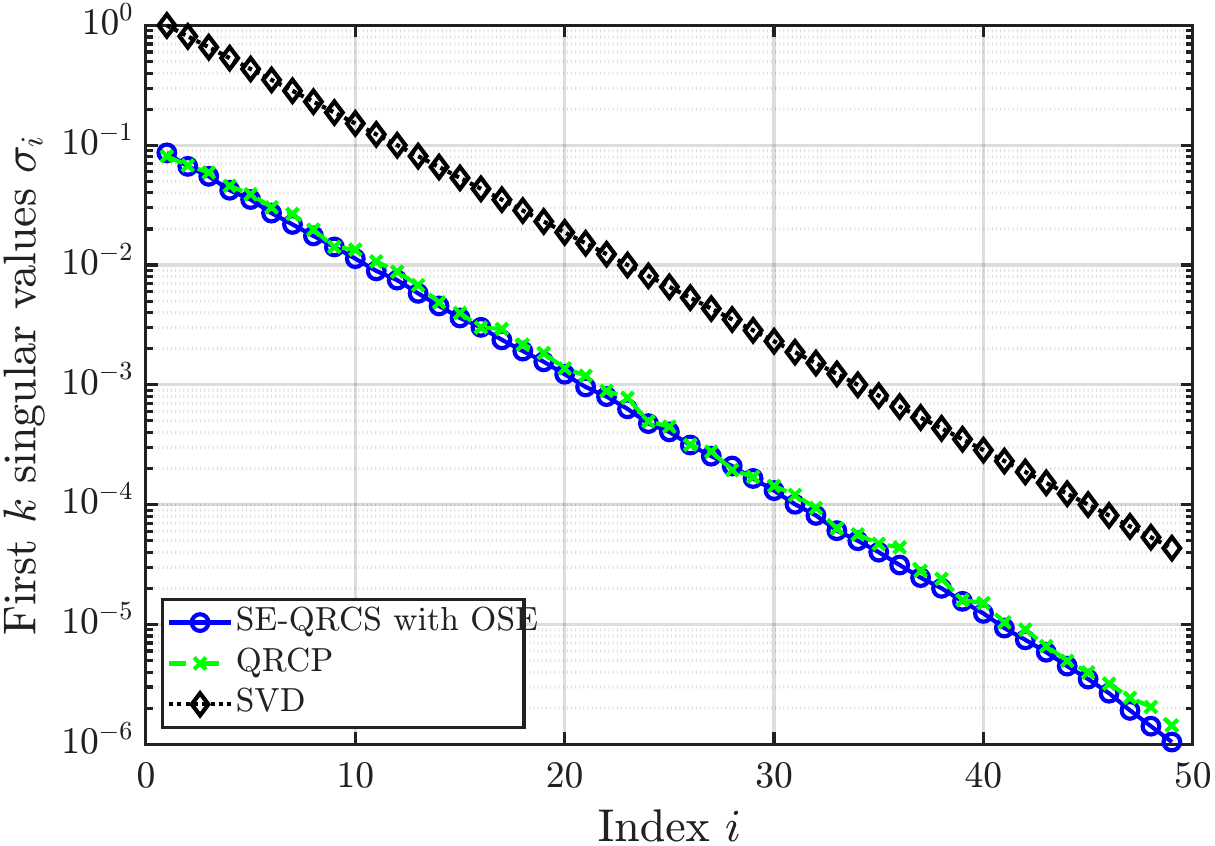"}
        \caption{} 
        \label{fig:sing_exponential}
    \end{subfigure}
    \hfill
    \begin{subfigure}[b]{0.48\textwidth}
        \centering
        \includegraphics[width=\textwidth]{"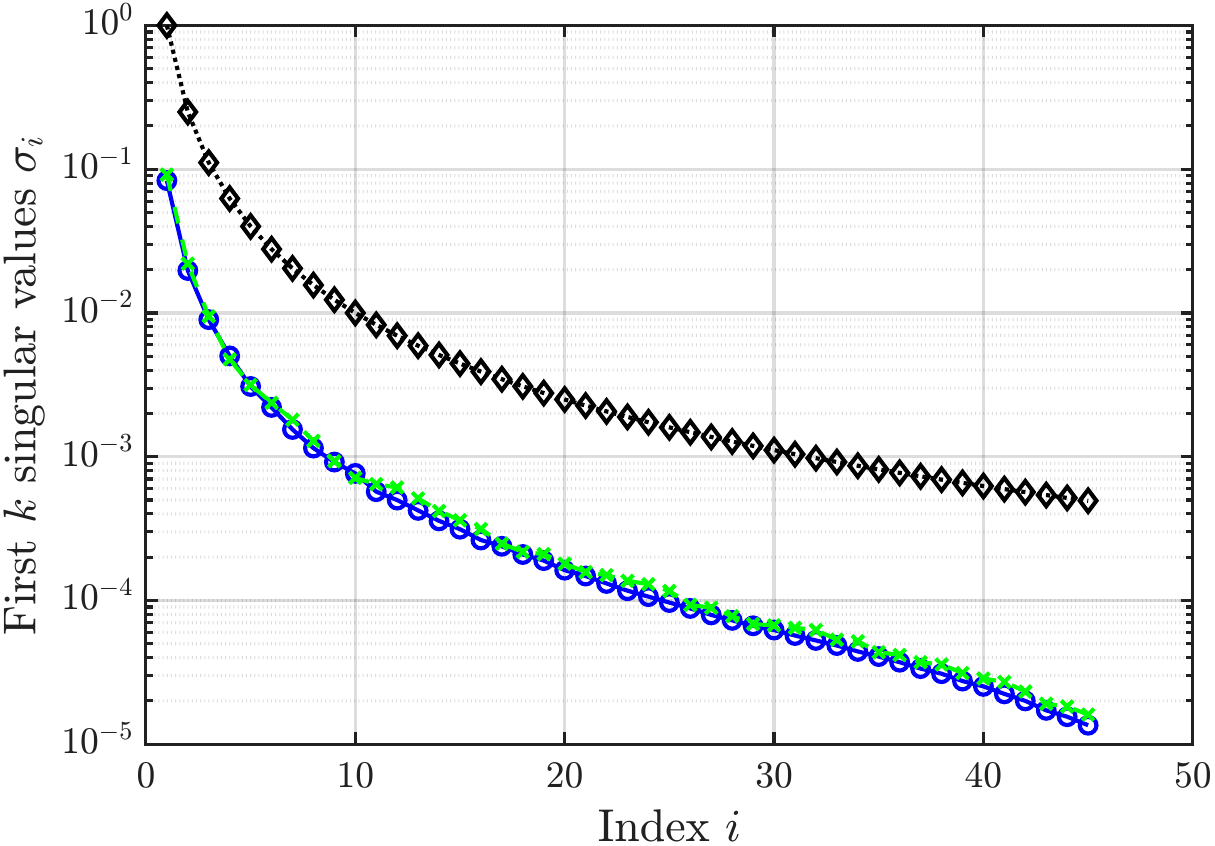"}
        \caption{}
        \label{fig:sing_quadratic}
    \end{subfigure}
    
    \vspace{0.3cm} 
    
    \begin{subfigure}[b]{0.48\textwidth}
        \centering
        \includegraphics[width=\textwidth]{"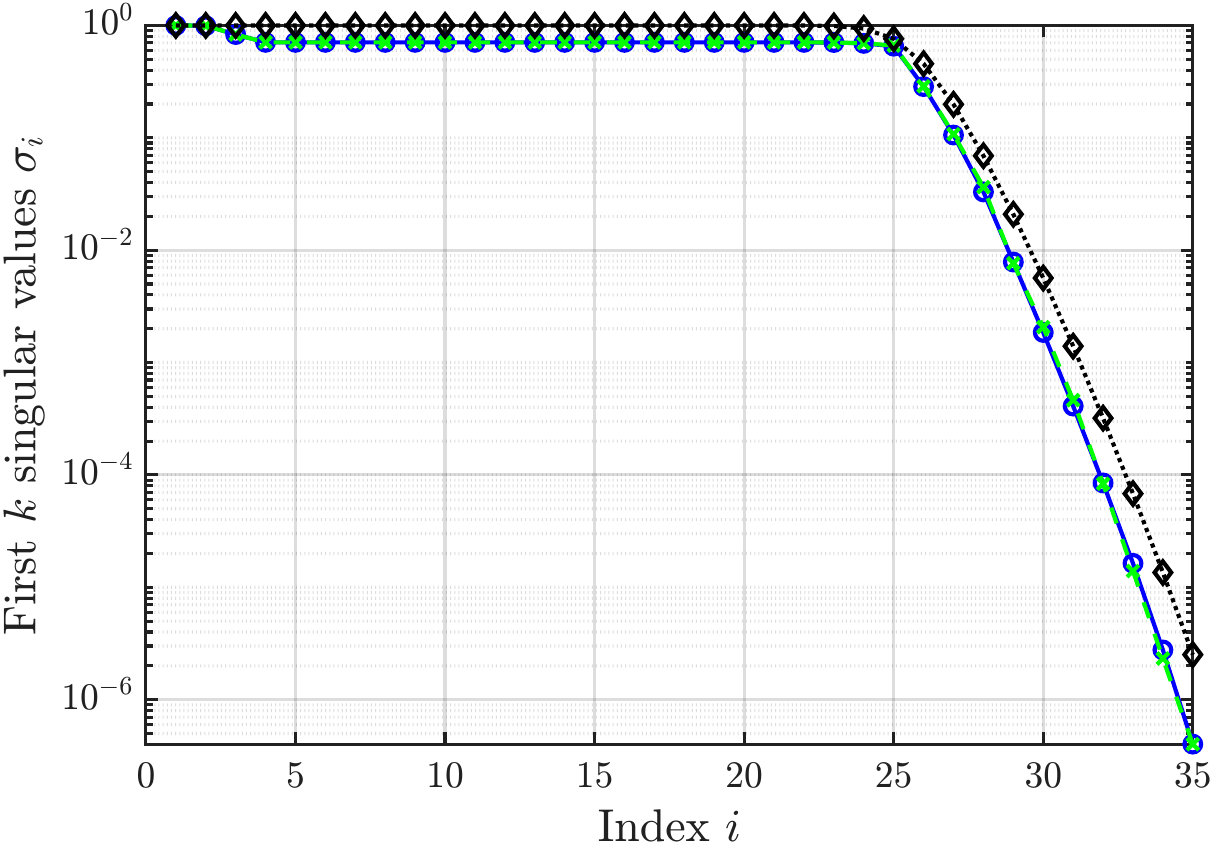"}
        \caption{}
        \label{fig:sing_prolate}
    \end{subfigure}
    \hfill
    \begin{subfigure}[b]{0.48\textwidth}
        \centering
        \includegraphics[width=\textwidth]{"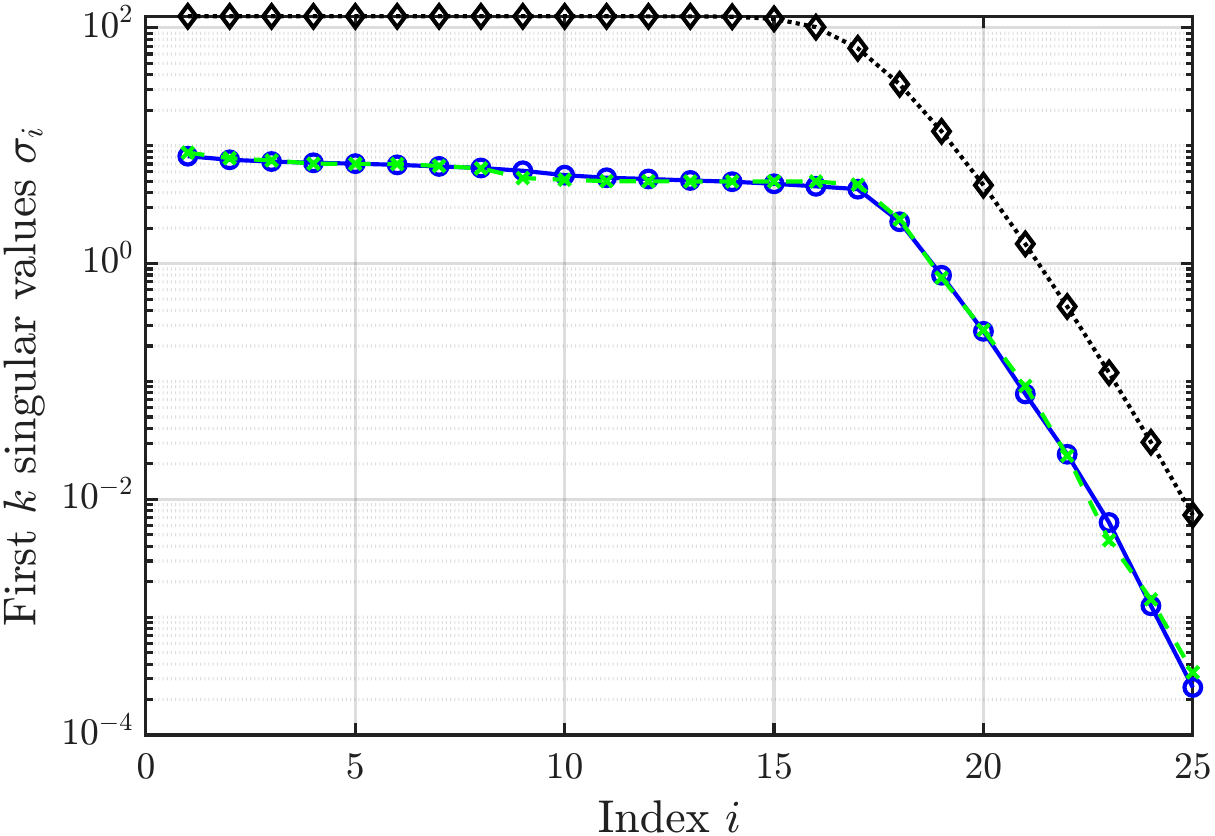"}
        \caption{}
        \label{fig:sing_chebvand}
    \end{subfigure}
    
    \caption{The figures show the first $k$ singular values $\sigma_{i}(R_{11})$ obtained across four matrix examples of size $50 \times 10000$: (a) exponential decay matrix, (b) quadratic decay matrix, (c) Prolate matrix, and (d) Chebvand matrix. Results are compared across the true singular values from exact SVD (black dotted line with diamond markers), deterministic QRCP (green dashed line with cross markers), and SE-QRCS using a CountSketch matrix with $s=1$ (solid blue line with circle markers).}
    \label{fig:results_s1}
\end{figure}
\begin{figure}[H]
    \centering
    \includegraphics[width=\linewidth]{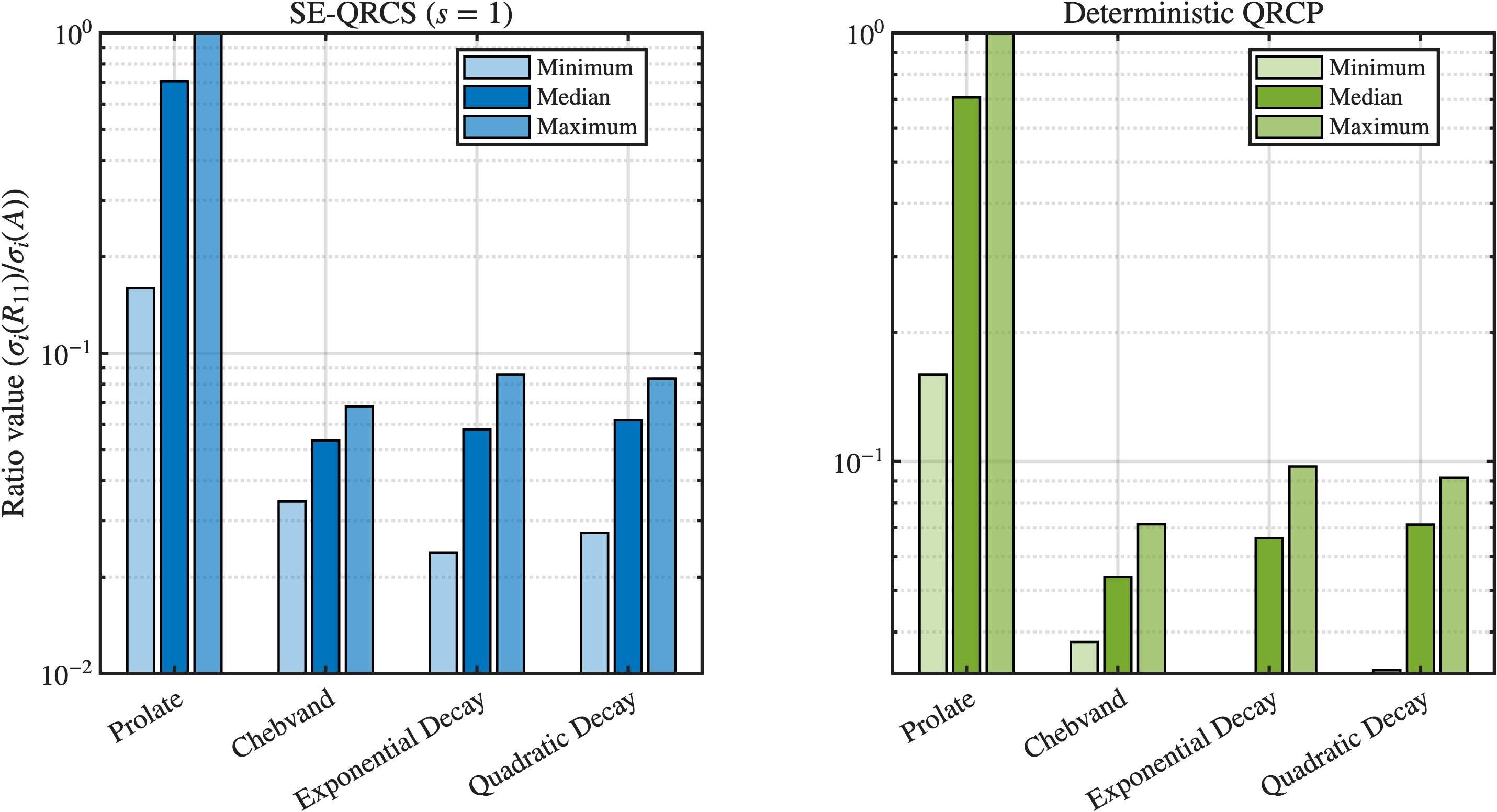}
    \caption{The grouped bar charts display the minimum, median, and maximum ratios of $\sigma_i(R_{11})/\sigma_i(A)$ plotted on a logarithmic scale. Panels show a comparison between: (a) SE-QRCS with an Oblivious Sparse Embedding (OSE, blue) and  (b) the deterministic QRCP baseline (green). The matrices used are of size $50 \times 10000$ }
    \label{fig:placeholder}
\end{figure}
The ranks of the matrices in Table \ref{tab:matrix_list} are chosen so that the relative spectral residual error is as stated in Table \ref{tab:residual error}. We also compare it with the relative spectral residual error, $\frac{\|R_{22}\|_{2}}{\|A\|_{2}}$, of QR with column pivoting. These results show that the residual error of \RandWSQR{} is nearly the same as that of QRCP, indicating that the two methods achieve comparable accuracy.
\begin{table}[H]
    \centering
    \begin{tabular}{|c|c|c|}
        \hline
        \textbf{Type of Matrix} & \textbf{QRCP $\left(\frac{\|R_{22}\|_{2}}{\|A\|_{2}}\right)$} & \textbf{\RandWSQR{} $\left(\frac{\|R_{22}\|_{2}}{\|A\|_{2}}\right)$} \\ \hline
         Matrix with exponential spectral decay &   6.0141e-05
         & 6.8971e-05 \\ \hline
        Matrix with quadratic spectral decay & 7.1895e-04 &8.3845e-04\\ \hline
        Fiedler matrix &   2.0075e-07 &  2.0075e-07 \\ \hline
        ROM &  1.4800e-02 &     1.4800e-02 \\ \hline
        Prolate &     2.1832e-06 &   1.4889e-06\\ \hline
        Chebvand &2.2023e-05 & 2.2906e-05\\ \hline
        Random low rank matrix & 3.2386e-16 & 1.0847e-15 \\ 
  \hline
    \end{tabular}
    \caption{Comparison of the relative spectral residual error between \RandWSQR{} and QRCP for several matrices in our test set. }
    \label{tab:residual error}
\end{table}
\subsection{\RandWSQR{} using oblivious sparse embeddings with $\sss>1$ and LESS Embeddings}
\label{experiments_s>1}
 This section presents the results obtained when using oblivious sparse embeddings with $s>1$ and Leverage Score Sparsified (LESS) embeddings for four matrices: three of dimension $200\times 10000$( the random outlier matrix, the random low rank matrix, and an ill conditioned matrix, the Fiedler matrix) and Genomic dataset matrix of size $400\times 239322$. For the OSE framework, the embedding dimension is set to $\ell = \left\lfloor 2 \cdot \sM \cdot \log(\sM) \right\rfloor$ with sparsity parameter $s=6$. For the LESS framework, we use \texttt{LESS-IND-ROWS} variant with parameter $q=0.1$ \cite{chenakkod2024optimal}. Figures \ref{fig:results_s>1}(a)--\ref{fig:results_s>1}(d) show the singular values of $R_{11}$, $\sigma_{i}(R_{11})$ and the singular values of $A$, $\sigma_{i}(A)$, with $R_{11}$ resulting from \RandWSQR{} and QRCP respectively. The table in Figure \ref{fig:results_s>1} shows the median of $p$, which is the corresponding size of $\tilde{A}_{1}$, across these four test cases for both OSE and LESS embeddings. Across all test matrices, LESS performs similarly to OSE but has significantly smaller size $p$. Regarding the singular value approximation, these results show that in both the random outlier and Fiedler case, the singular values $\sigma_{i}(R_{11})$ for both \RandWSQR{} (under OSE and LESS configurations) and QRCP are nearly identical to the corresponding singular values $\sigma_{i}(A)$. This can also be seen in Figure \ref{fig:summary_barcharts} where the minimum, maximum, and median of the ratios $\sigma_{i}(R_{11})/\sigma_{i}(A)$ for all three factorizations are close to $1$. In the random low rank matrix, the singular values $\sigma_{i}(R_{11})$ obtained from \RandWSQR{} are not close to the exact singular values, but still align with QRCP.
\begin{figure}[htbp] 
    \centering
    
    \begin{subfigure}{0.45\linewidth} 
        \centering
        \includegraphics[width=\linewidth]{"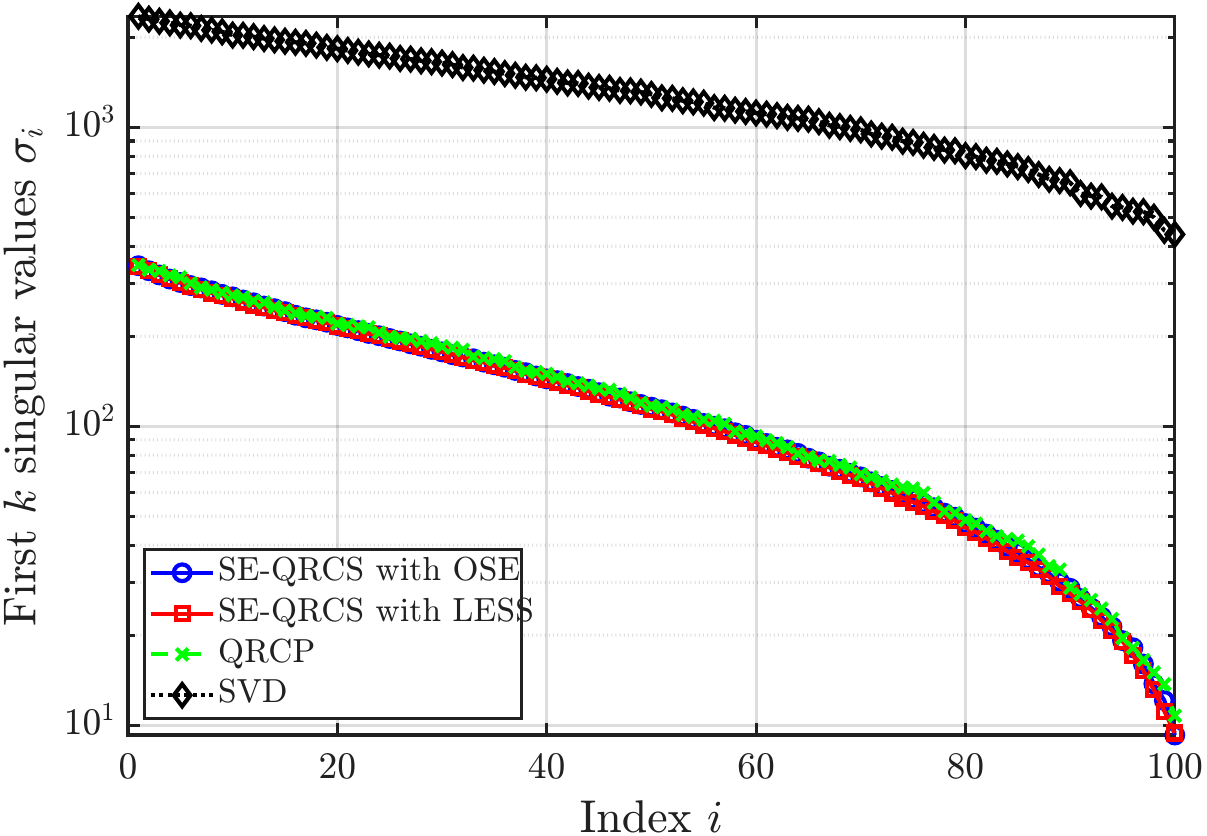"} 
        \caption{Random Low-Rank Matrix}
        \label{fig:lowrank_sing}
    \end{subfigure}
    \hfill 
    \begin{subfigure}{0.45\linewidth} 
        \centering
        \includegraphics[width=\linewidth]{"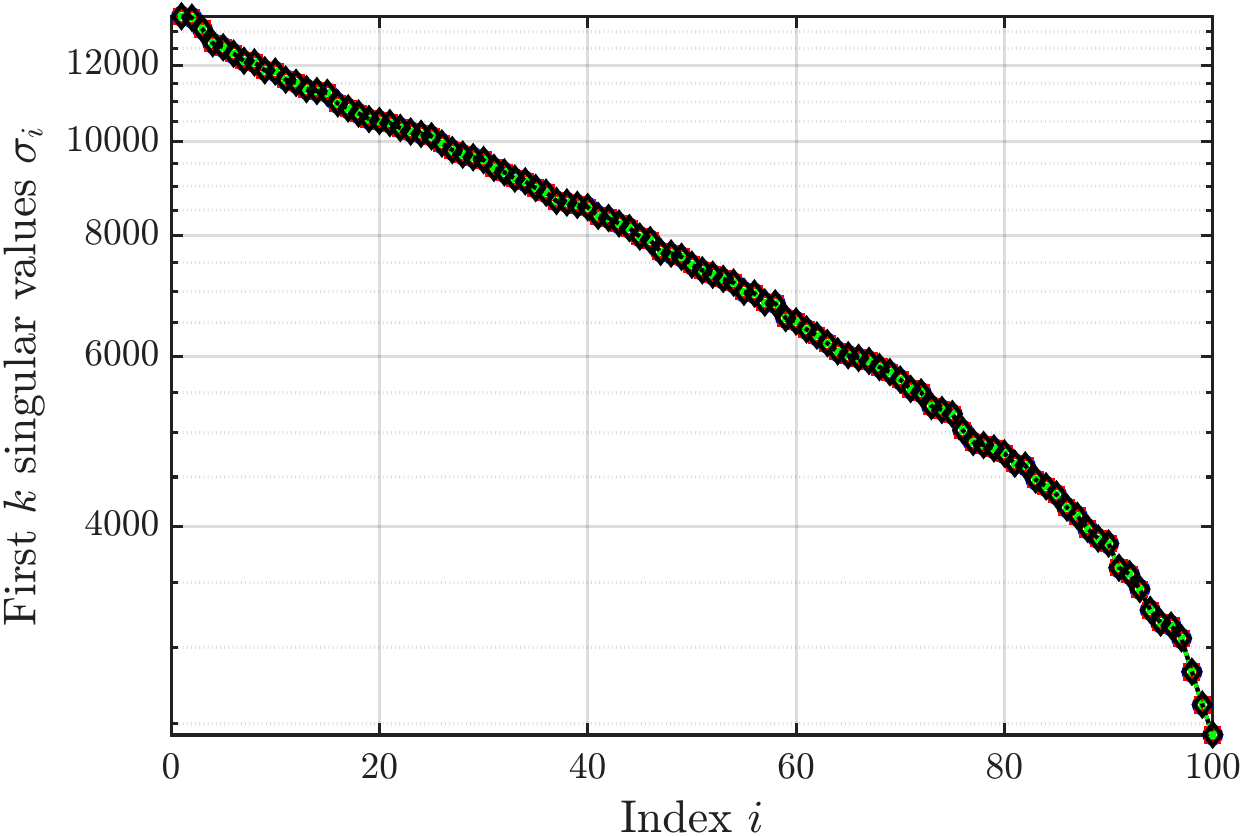"}
        \caption{Random Outlier Matrix}
        \label{fig:outlier_sing}
    \end{subfigure}
    
    \vspace{0.5cm} 
    
    \begin{subfigure}{0.47\linewidth} 
        \centering
        \includegraphics[width=\linewidth]{"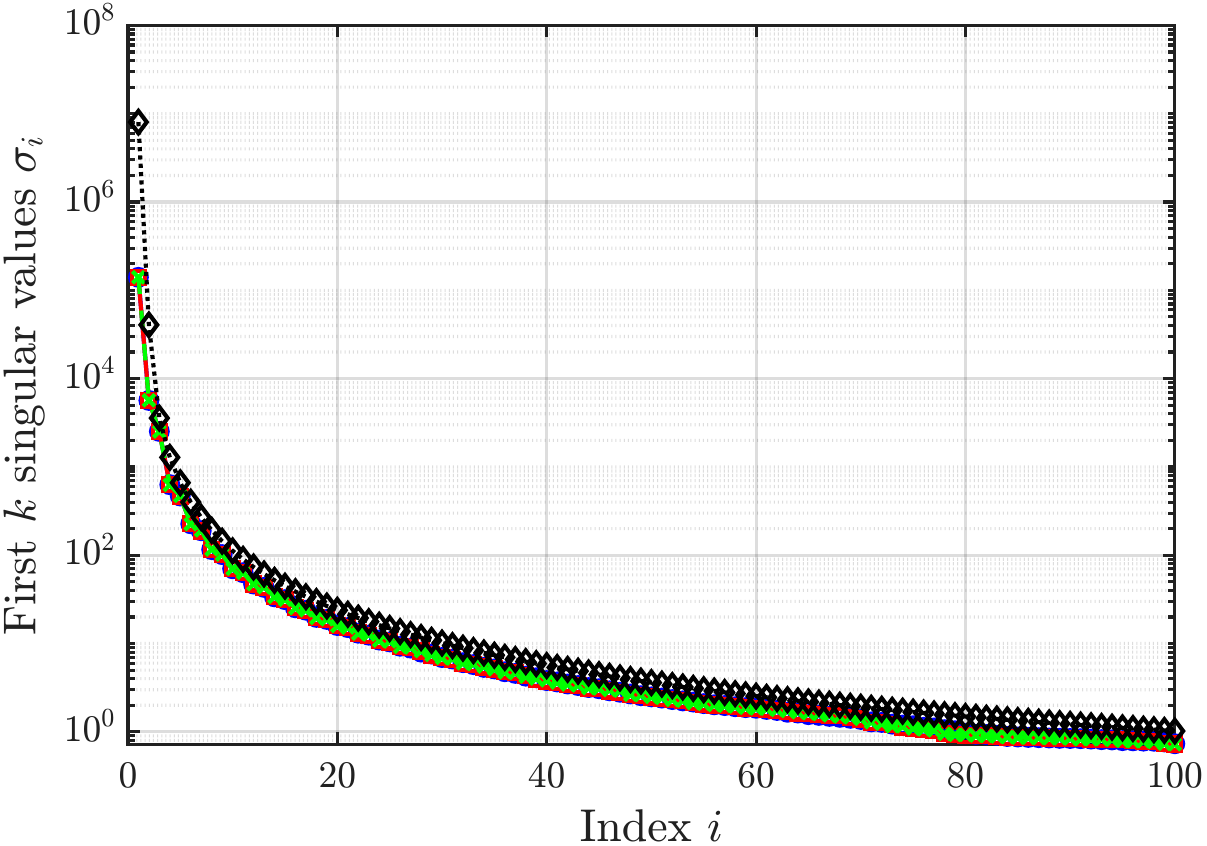"}
        \caption{Fiedler Matrix}
        \label{fig:fiedler_sing}
    \end{subfigure}
    \hfill
    \begin{subfigure}{0.47\linewidth}
        \centering
        \includegraphics[width=\linewidth]{"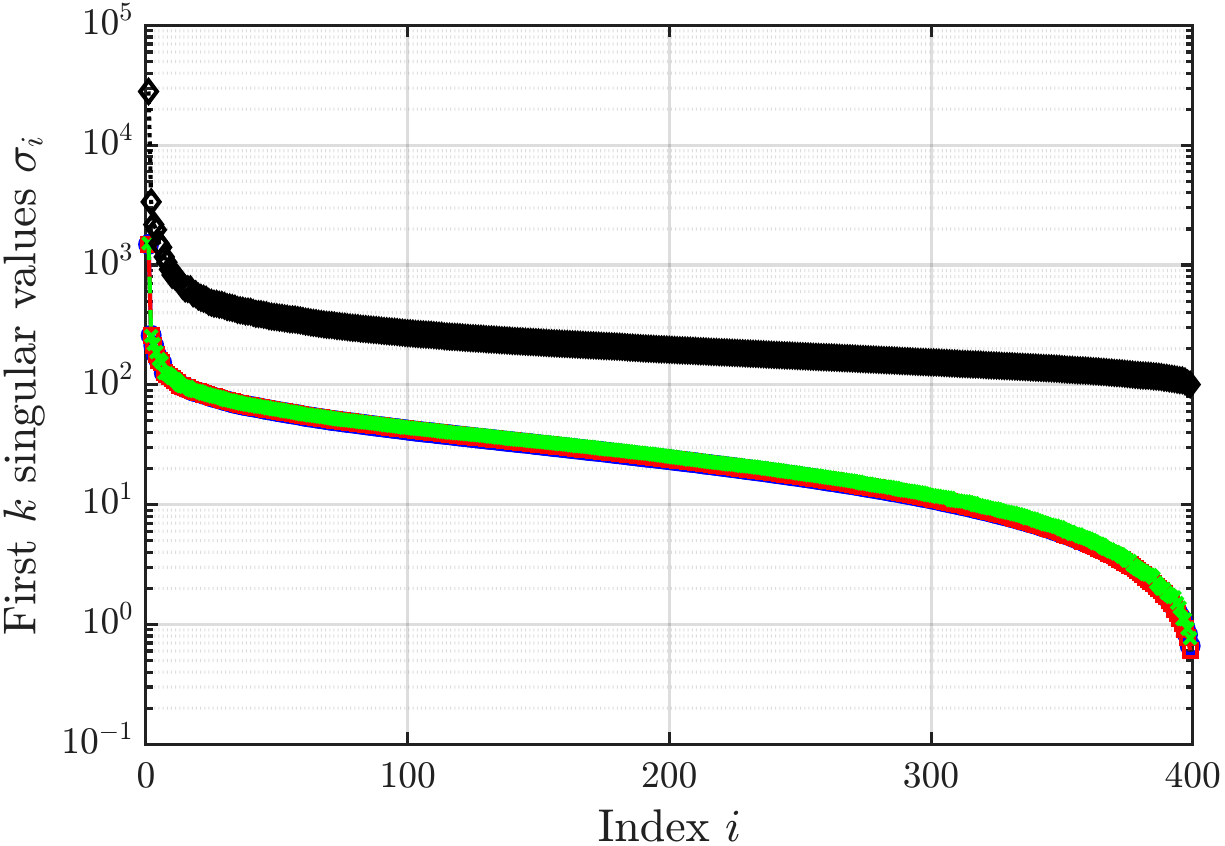"}
        \caption{Genomic Dataset}
        \label{fig:matrix4_sing}
    \end{subfigure}
    
    \vspace{0.6cm} 
    \small
    \begin{tabular}{lcc}
        \hline
        \textbf{Test Matrix ($m \times n$)}  & \textbf{Median $p$ (Oblivious)} & \textbf{Median $p$ (LESS)} \\ \hline
        Random Outlier ($200 \times 10000$) & 2609 & 780 \\
        Random Low Rank ($200 \times 10000$) &  3069 &  1816 \\
        Fiedler Matrix ($200 \times 10000$)   & 2648 & 356 \\
        Genomic Dataset ($400 \times 239322$)        &   6992 & 14226 \\ \hline
    \end{tabular}

    \vspace{0.4cm}

    \caption{Figures~\ref{fig:lowrank_sing}, \ref{fig:outlier_sing}, \ref{fig:fiedler_sing} and \ref{fig:matrix4_sing} plot the median values over 20 trials for the first $k=100$ and $k=400$ singular values $\sigma_i(R_{11})$ or $\sigma_i(A)$ evaluated against SVD (black dotted line with diamond markers), deterministic QRCP (green dashed line with cross markers) , and \RandWSQR{} using Oblivious Sparse Embeddings OSE( solid blue line with circles) or LESS Embeddings (solid red line with squares). The table tracks the corresponding median selected column count $p$ across each individual trial block.}
    \label{fig:results_s>1}
\end{figure}

\begin{figure}[htbp] 
    \centering
    \includegraphics[width=\linewidth]{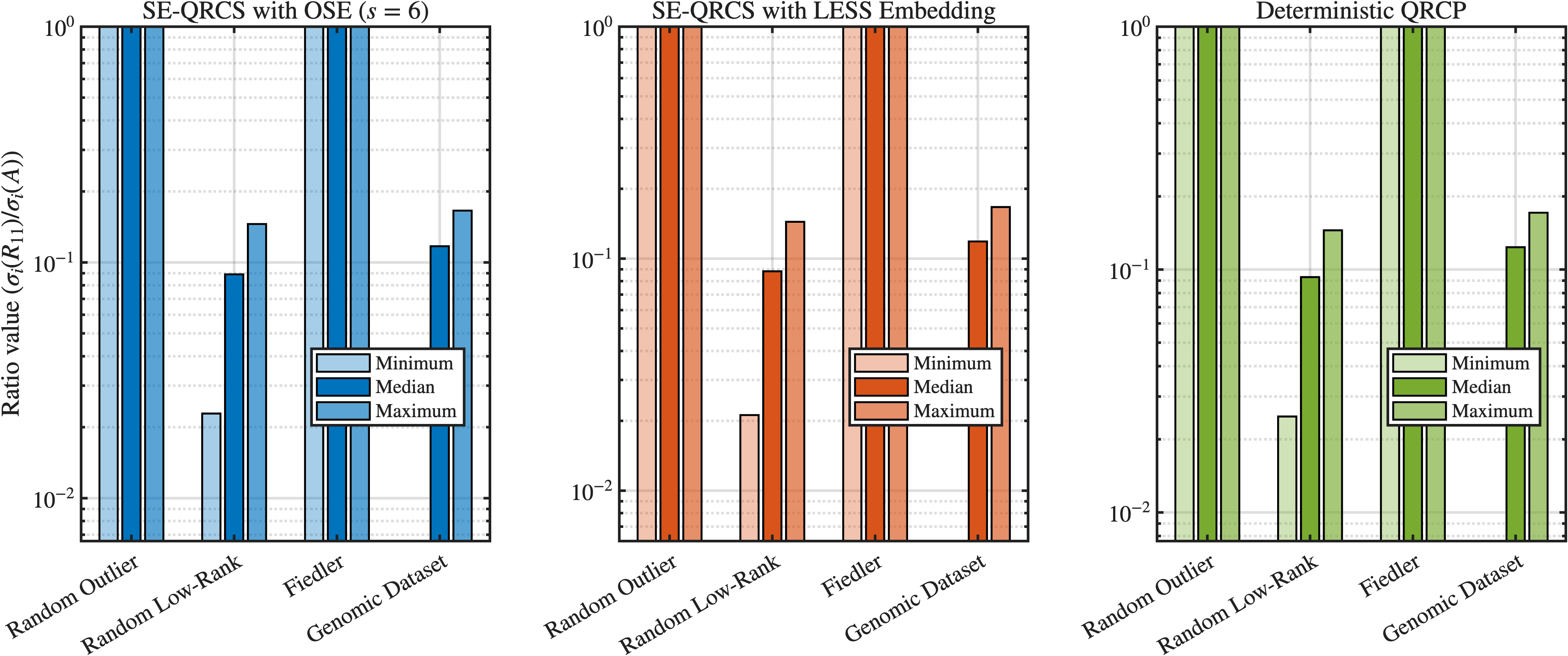}
    \caption{The grouped bar charts display the minimum, median, and maximum ratios of $\sigma_i(R_{11})/\sigma_i(A)$ plotted on a logarithmic scale. Panels show a comparison between: (a) SE-QRCS with an Oblivious Sparse Embedding (OSE, blue), (b) SE-QRCS with a LESS Embedding (red), and (c) the deterministic QRCP baseline (green). }
    \label{fig:summary_barcharts}
\end{figure}

\subsection{Comparison for different sketching dimension and sparsity}
\label{comparison subsection}
In Figures~\ref{fig:vs_l}--\ref{fig:vs_s}, we evaluate the performance of 
\RandWSQR{} using an Oblivious Subspace Embedding (OSE) across various embedding dimensions~$\ell$ and sparsity parameters~$s$. For these experiments, we use ROM matrix of size 
$100 \times 10^{6}$ containing $40$ amplified columns. 
The numerical results show that varying $\ell$ and $s$ yields 
comparable singular value ratios. This can be tied to the 
structural trade-off between the sketching dimension~$\ell$ and the column 
dimension~$p$ of the  matrix~$\tilde{A}_{1}$. Specifically, as the 
sketching dimension increases, the number of selected columns~$p$ decreases, 
and vice versa. This trade-off also explains the subtle variations observed 
in the runtime. We note that the computational overhead scales more significantly with larger values of~$p$ than with~$\ell$, as the larger column dimension dominates the execution time of the algorithm.
\begin{figure}[H]
    \centering
    \includegraphics[width=\linewidth]{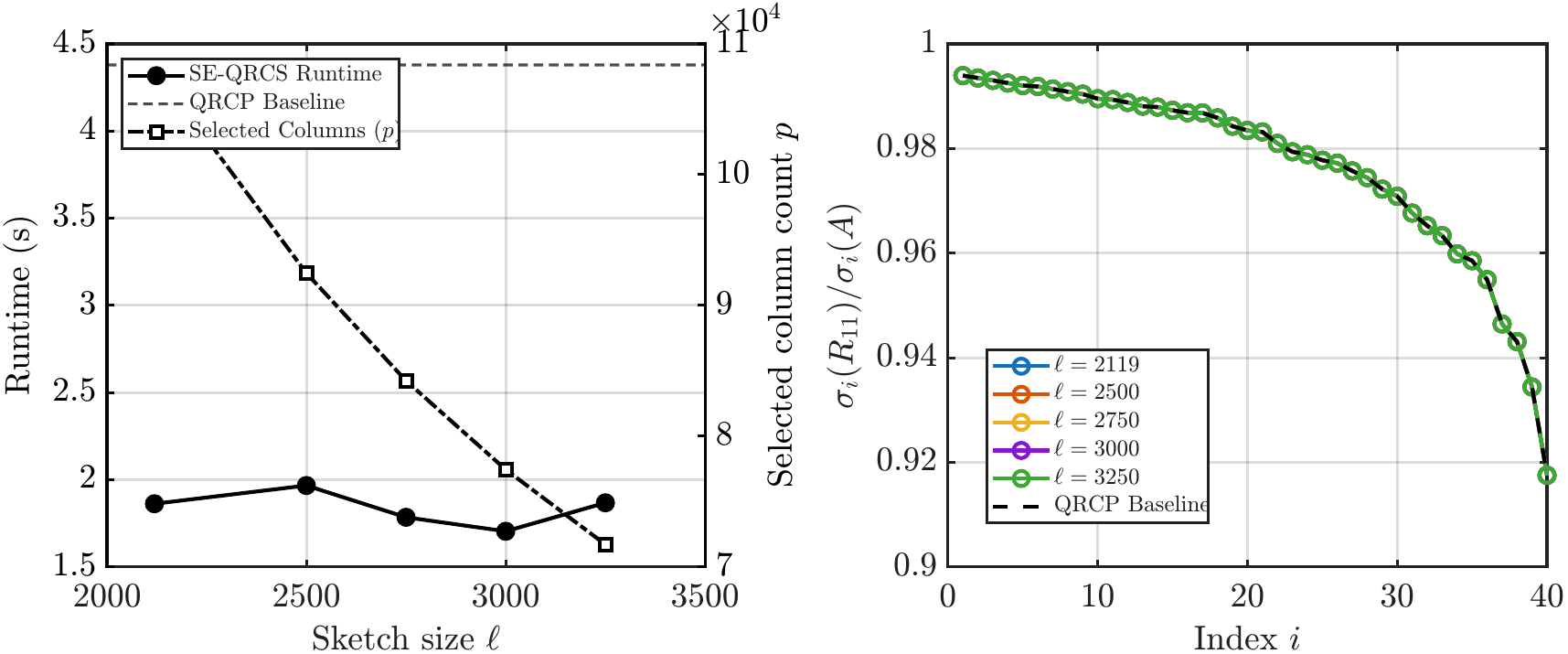}
  \caption{Performance metrics for the Random Outlier Matrix (ROM) as a function of sketching size $\ell$: (Left) variation of runtime and size of $\tilde{A}_{1}$, $p$, and (Right) comparison of the first $k$ singular value ratios between \RandWSQR{} and QRCP.}
    \label{fig:vs_l}
\end{figure}
\begin{figure}[h!]
    \centering
    \includegraphics[width=\linewidth]{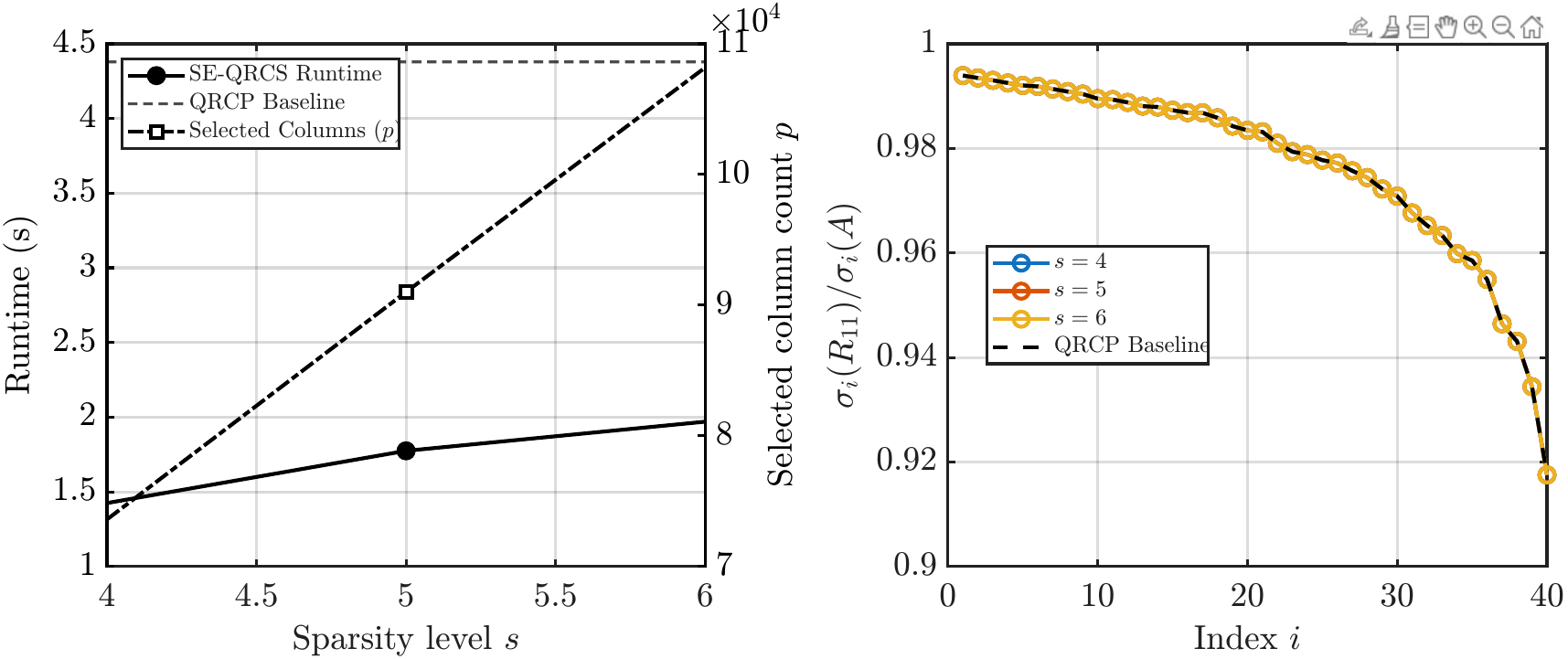}
  \caption{Performance metrics for the Random Outlier Matrix (ROM) as a function of sparsity level $s$: (Left) variation of runtime and size of $\tilde{A}_{1}$,$p$, and (Right) comparison of the first $k$ singular value ratios between \RandWSQR{} and QRCP.}
    \label{fig:vs_s}
\end{figure}

\subsection{Computational complexity and runtime}
\label{computational time}
To evaluate the reduction in computational complexity and thus runtime, Figures \ref{figure 26} and \ref{figure 27} show the execution time for pivot selection using \RandWSQR{} and QRCP factorization. In this experiment, the matrices are of size $100\times  10^{6}$, and a CountSketch embedding is used in the \RandWSQR{} case, with ranks 70 and 100 for the matrix with exponential spectral decay and the random matrix, respectively. The results indicate a reduction in runtime by factors of 7 and 10 for the random matrix and the matrix with exponential spectral decay, respectively. Examining the time taken for QRCP on $B$ and $\tilde{A}_{1}$ alone shows a speedup by factor of $87$. With a more efficient implementation of sketching using sparse embeddings, even greater speedups are possible. Figure~\ref{fig:Less_vs_OSE_time} illustrates the runtime comparison between \RandWSQR{} utilizing a CountSketch matrix versus a LESS embedding across varying matrix dimensions $100 \times n$. The empirical results show that the LESS embedding achieves superior efficiency, with its performance advantage becoming increasingly pronounced as the matrix size $n$ scales.

   \begin{figure}[H]
    \centering
    \begin{minipage}[t]{0.40\linewidth} 
        \captionsetup{font=footnotesize}
        \centering
        \includegraphics[width=\linewidth]{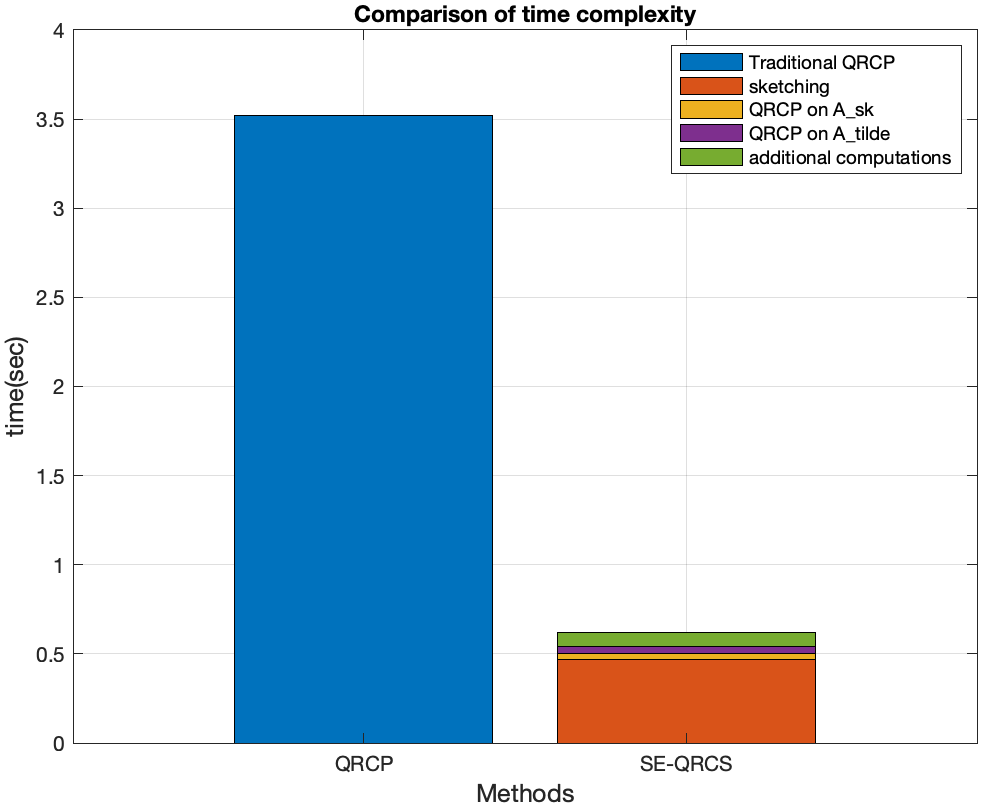}
        \caption{Random Matrix}
        \label{figure 26}
    \end{minipage}
    \hspace{0.02\linewidth} 
    \begin{minipage}[t]{0.40\linewidth}
        \captionsetup{font=footnotesize}
        \centering
        \includegraphics[width=\linewidth]{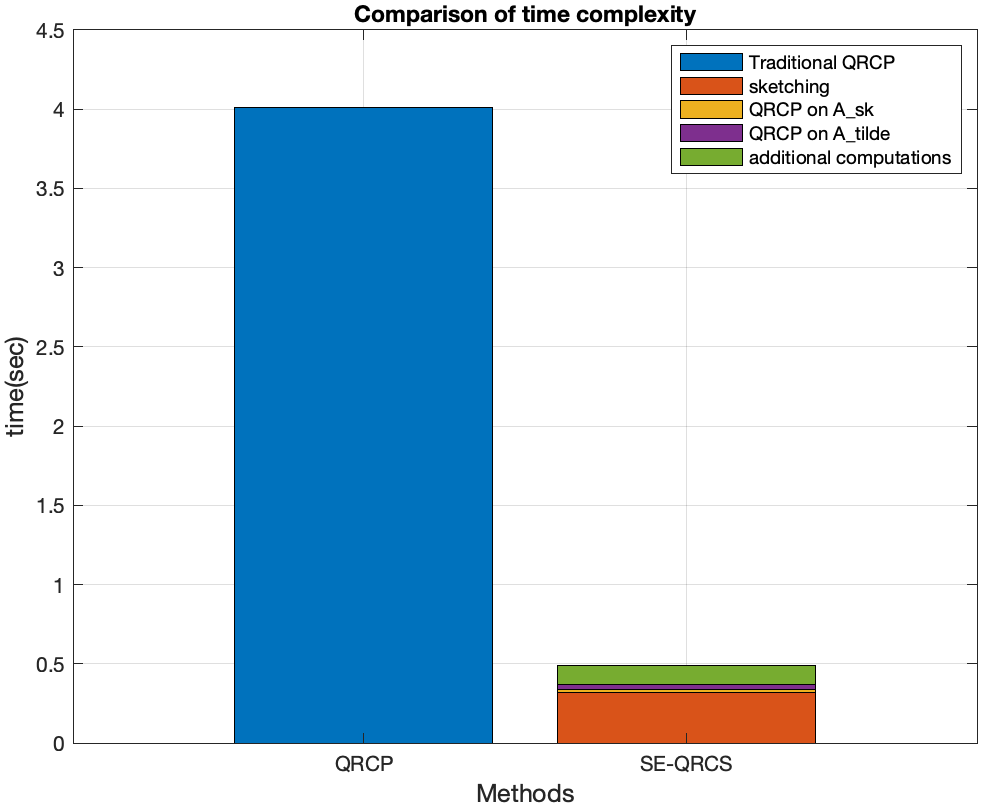}
        \caption{Matrix with exponential spectral decay}
        \label{figure 27}
    \end{minipage}
    \label{fig:ratio_comp_s=1}
\end{figure}

\begin{figure}[H]
    \centering
    \includegraphics[width=0.5\linewidth]{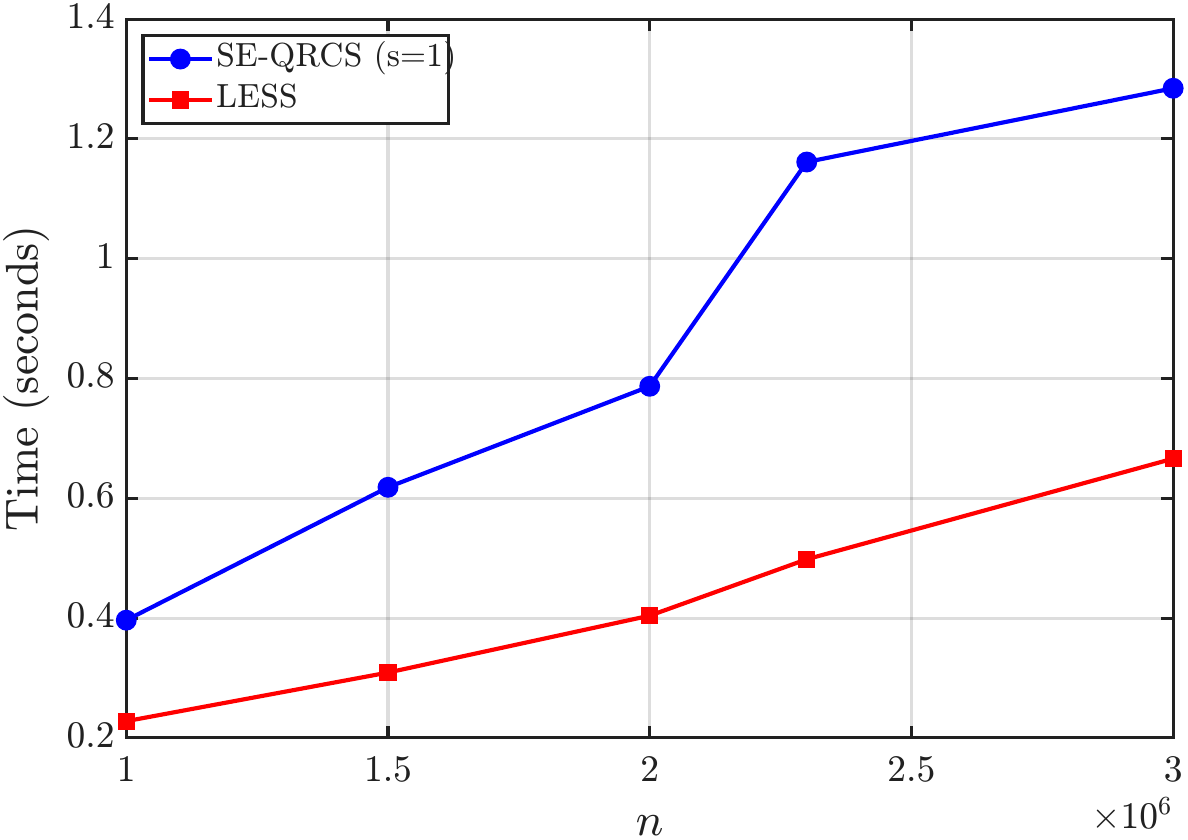}
    \caption{Comparison of time between CountSketch embedding and LESS on random matrix of different sizes $100\times n$ with rank $k=60$. For CountSketch embedding $\ell=10000$ and $\ell=1000$ for LESS. The time compared corresponds to SE-QRCS without the addition of time for computing leverage scores in LESS case.}
    \label{fig:Less_vs_OSE_time}
\end{figure}

\subsection{LU\_PRRP using \RandWSQR{}}
 In this section, we introduce one application of \RandWSQR{} on the LU\_PRRP algorithm presented in \cite{khabou2013lu}. In this algorithm, LU factorization algorithm is performed based on strong rank revealing QR panel factorization. The matrix $A\in \mathbb{R}^{n\times n}$ is divided into panels of size $b$. At each iteration, sRRQR is applied to the transpose of the current panel $A^{(i)}$ and the obtained permutation of the rows is applied to the matrix $A$. After that, sRRQR is applied to the first block $b\times b$ in the panel, and similarly the obtained row permutation is applied to the matrix $A$. The trailing matrix is then updated. They show in \cite{khabou2013lu} that the resulting factorization has better stability than the traditional LU factorization with growth factor $g_{w}=(1+fb)^{n/b-1}$, where $f$ is the strong rank revealing QR constant. In particular, it produces accurate results on certain matrices where Gaussian elimination with partial pivoting (GEPP) fails, such as the Wilkinson matrix \cite{wilkinson1961error}. 

LU\_PRRP factorization can be accelerated by applying \RandWSQR{} to the transpose of the tall and skinny panels, i.e. $A^{(i)}\in \mathbb{R}^{w\times b}$ with $w>2b$. This reduces the algorithm's complexity while still providing a good growth factor. For example, it works on the Wilkinson matrix and provides a small growth factor, provided that the block size $b$ is not chosen very small, since the bound on $|R_{11}^{-1}R_{12}|$ in \RandWSQR{} factorization is greater than that of sRRQR.

Table \ref{Table2:application} presents the results of running LU\_PRRP with \RandWSQR{} on random matrices of varying dimensions and block sizes. It reports the growth factor, $\|U\|_{1}$ and the Frobenius error; these are compared against the  GEPP results in Table \ref{Table1:application}. The values indicate an improvement in the growth factor for block sizes $b\geq 16$, with only a slight increase for $b=8$. We also evaluated the performance over some special matrices, including ill conditioned cases and the Wilkinson matrix, as shown in Tables \ref{Table3:application} and \ref{Table4:application}.
\begin{table}[htbp]
\centering
\begin{tabular}{|c|c|c|c|c|}
\hline
n & growth factor & $\|U\|_{1}$ & $\|U^{-1}\|_{1}$ & $\frac{\|PA-LU\|_{F}}{\|A\|_{F}}$ \\ \hline
8192 & 56.6869 & 1.5931e+05 & 424.1910 & 4.5815e-14 \\ \hline
4096 & 35.9525 & 5.9614e+04 & 245.1152 & 2.3454e-14 \\ \hline
2048 & 22.8242 & 2.2521e+04 & 79.4674 & 1.2246e-14 \\ \hline
1024 & 18.0337 & 78.4041 & 8.3458e+03 & 5.9966e-15 \\ \hline
\end{tabular}
\caption{Results of applying GEPP on different random matrices.}
\label{Table1:application}
\end{table}
\begin{table}[htbp]
\centering
\begin{tabular}{|c|c|c|c|c|c|}
\hline
n & b & growth factor &  $\|U\|_{1}$ & $\|U^{-1}\|_{1}$ & $\frac{\|PA-LU\|_{F}}{\|A\|_{F}}$ \\ \hline
\multirow{4}{*}{8192} & 64 & 32.2024 & 2.0116e+05 & 301.1565
 & 1.4381e-13
\\ \cline{2-6}
 & 32 &41.4839 & 2.0467e+05&  196.3066 &  1.0930e-13
 \\ \cline{2-6}
& 16 & 56.0013 & 2.0846e+05 &  258.1333 & 8.7152e-14
 \\ \cline{2-6}
& 8 & 60.8820 & 1.9780e+05 &284.0822 & 6.8535e-14 \\ \hline
\multirow{4}{*}{4096} & 64 & 21.6249 & 7.2546e+04& 424.1910 & 7.7317e-14 
\\ \cline{2-6}
& 32 & 30.0180 & 7.2769e+04 & 379.1695 & 5.8879e-14
 \\ \cline{2-6}
& 16 & 33.8969 & 7.7480e+04 &378.0365 & 4.7525e-14 \\ \cline{2-6}
& 8 & 36.91 & 7.5933e+04 & 245.1152 & 3.5475e-14 \\ \hline
\multirow{4}{*}{2048} & 64 & 12.9146 & 2.6417e+04 & 78.7006 & 3.9019e-14 
\\ \cline{2-6}
& 32 & 16.4287 &  2.7682e+04 & 78.9232 & 3.1959e-14 
\\ \cline{2-6}
& 16 & 28.5681 & 2.7684e+04 & 101.1060 & 2.6362e-14 
\\ \cline{2-6}
& 8 & 32.0041 & 2.9034e+04 & 98.5623 & 2.0639e-14 \\ \hline
\multirow{4}{*}{1024} & 64 & 10.4042 & 9.4457e+03 & 55.0783 & 1.8277e-14 
\\ \cline{2-6}
& 32 & 11.9826 & 9.8460e+03 & 55.0783 & 1.7520e-14
\\ \cline{2-6}
& 16 & 16.7384 & 1.0334e+04 &  89.0014 &  1.4911e-14 
\\ \cline{2-6}
& 8 & 21.5211 & 1.0400e+04 &139.4189 &  1.1525e-14 \\ \hline
\end{tabular}
\caption{Results of applying LU-PRRP using \RandWSQR{} on different random matrices}
\label{Table2:application}
\end{table}
\begin{table}[htbp]
\centering
\begin{tabular}{|c|c|c|c|c|}
\hline
Type & growth factor & $\|U\|_{1}$ & $\|U^{-1}\|_{1}$ & $\frac{\|PA-LU\|_{F}}{\|A\|_{F}}$ \\ \hline
Fiedler & 2 & 8.9979e+06 & 2.3760 & 1.1163e-15 \\ \hline
Chebvand & 319.8084 & 9.1629e+17 & 1.0198e+05 & 0.6387e-13 \\ \hline
Prolate & 29.5330 &  3.5291e+03 & 1.1522e+21 & 0.3576e-13 \\ \hline
Condex & 1 & 602.3509 & 1.2946 & 0.0652e-14\\ \hline
Circul & 1& 4.0503e+05 & 0.5489e-03 & 0.0682e-14\\ \hline
Wilkinson & 4.4305& 1.0249e+03 & 194.0805 & 0.5484e-15\\ \hline
\end{tabular}
\caption{Results of applying LU\_PRRP using \RandWSQR{} $(\text{with }s=1)$ on different special matrices.}
\label{Table3:application}
\end{table}
\begin{table}[htbp]
\centering
\begin{tabular}{|c|c|c|c|c|}
\hline
Type & growth factor & $\|U\|_{1}$ & $\|U^{-1}\|_{1}$ & $\frac{\|PA-LU\|_{F}}{\|A\|_{F}}$ \\ \hline
Fiedler & 1.9995 & 1.6769e+07 & 2 & 4.5952e-15 \\ \hline
Chebvand & 202.4801 & 1.3953e+17 & 0.7487e+05 & 0.3045e-13 \\ \hline
Prolate & 20.5850 & 2.1260e+03 & 0.0075e+21 & 0.1060e-13 \\ \hline
Condex & 1 & 602.3509 & 1.2946 & 0.0652e-14 \\ \hline
Circul & 1& 0.0836e+05 & 0.5459e-03 & 0.1444e-14\\ \hline
\end{tabular}
\caption{Results of applying GEPP on different special matrices.}
\label{Table4:application}
\end{table}

%% file: conclusion.tex
\section{Conclusion}
This work presents \RandWSQR{} factorization, a method for selecting $k$ columns from a matrix to capture its spectrum or to construct a low rank approximation. This approach employs sparse embeddings to reduce computational costs, particularly for matrices with many more columns than rows. It is shown that this factorization satisfies the strong rank revealing QR properties. Future work could explore extending this algorithm to tensor factorizations or using a different matrix factorization other than sRRQR.